\documentclass[review]{elsarticle}
\usepackage[colorlinks]{hyperref}
\usepackage{geometry} 
\usepackage{amsthm}
\usepackage{latexsym,url,amscd,amssymb}
\usepackage{amsmath,amsfonts}\numberwithin{equation}{section}
\usepackage{cleveref}
\usepackage{graphics}
\usepackage{epstopdf}
\usepackage{hyperref}
\usepackage{cases}
\usepackage{caption}
\usepackage{algorithm}
\usepackage{algpseudocode}
\usepackage{listings}
\usepackage{rotating}
\usepackage{framed}
\usepackage{booktabs}
\usepackage{CJK}
\usepackage{indentfirst}
\usepackage{bm}
\usepackage{float}
\usepackage{subfig}
\usepackage{array}
\usepackage{multirow}
\usepackage{natbib}
\usepackage{ulem}

\DeclareMathOperator*{\argmin}{arg\,min}

\def\A{{\mathcal A}}

\def\K{{\mathcal K}}
\def\L{{\mathcal L}}

\def\S{{\mathcal S}}
\def\T{{\mathcal T}}

\newtheorem{lemma}{Lemma}[section]
\newtheorem{definition}{Definition}[section]

\newtheorem{theorem}{Theorem}[section]

\begin{document}
\begin{frontmatter}

\title{An Efficient Method for Joint Delay-Doppler Estimation of Moving Targets in Passive Radar} %文章标题

\author[mymainaddress]{Mengjiao Shi}
\author[mymainaddress,mysecondaryaddress]{Yunhai Xiao\corref{mycorrespondingauthor}}
\cortext[mycorrespondingauthor]{Corresponding author}\ead{yhxiao@henu.edu.cn}
\author[mysecondaryaddress]{Peili Li}

\address[mymainaddress]{School of Mathematics and Statistics, Henan University, Kaifeng 475000, China}
\address[mysecondaryaddress]{Center for Applied Mathematics of Henan Province, Henan University, Zhengzhou 450046, China}

\begin{abstract}
 {Passive radar systems can detect and track the moving targets of interest by exploiting  non-cooperative illuminators-of-opportunity to transmit orthogonal frequency division multiplexing (OFDM) signals. 
These targets are searched  using a bank of correlators  tuned to the waveform corresponding to the given Doppler frequency shift and delay.
 In this paper, we study the problem of joint delay-Doppler estimation of moving targets in OFDM passive radar. 
 This task of estimation is described as an atomic-norm regularized convex optimization problem, or equivalently, a semi-definite programming problem.
The alternating direction method of multipliers (ADMM) can be employed which computes each variable in a Gauss-Seidel manner, but its convergence is lack of certificate.
In this paper, we use a symmetric Gauss-Seidel (sGS) to the framework of ADMM, which only needs to compute some of the subproblems twice but has the ability to ensure convergence.
We do some simulated experiments  which illustrate that the sGS-ADMM is superior to ADMM in terms of accuracy and computing time.}
\end{abstract}

\end{frontmatter}

\section{Introduction}
Passive radar equipped with orthogonal frequency division multiplexing (OFDM) signals has been widely used in  civilian and military targets' detection. Passive radar systems can detect and track a target of interest by
illuminating it with  illuminators of opportunity (IOs).  
Compared to active radar,  passive radar systems enjoy much superiority  due to the non-cooperative nature of IOs: (i) Transmitters are unnecessary so that the system is substantially smaller and less expensive; (ii) Operations through ambient communication signals which don't cause interference to the existing wireless communication signals, e.g., radio and television broadcasting signals \cite{howland2005fm, baczyk2011reconstruction,zhang2015joint}. Nevertheless, the implementations of passive radar might encounter some challenges, for example, the transmitted signal is generally unknown to the receiver because the IOs are not under control. In addition, the received direct-path signal is usually stronger than the target reflections, which makes it  difficult to detect and track the targets.
	 
To manage these challenges, many passive radar systems are equipped with an additional separate reference channel to collect the transmitted signal as a reference, so that, it can eliminate the unwanted echoes in surveillance channels, e.g. direct signals, clutter, and multi-path signals \cite{tao2010direct, colone2006cancellation, cardinali2007comparison}. 
At present, many wireless internet and mobile communications, like digital audio/video broadcasts, are all using   OFDM signals. 
The OFDM is a state-of-art scheme which can reduce  inter-symbol interference and can increase the bandwidth efficiency by using orthogonal characteristics between subcarriers.
In general, OFDM signals can also be generated efficiently by fast Fourier transform (FFT) \cite{ketpan2015target}. 
We note that the targets' detection methods with OFDM signals have been studied over the past few years, for instance, Sen \& Nehorai \cite{sen2010adaptive} proposed a method for detecting a moving target in the presence of multi-path reflection in adaptive OFDM radar. Palmer et al. \cite{palmer2012dvb} and Falcone et al. \cite{falcone2010experimental} estimated the delays and Doppler shifts of targets when the demodulation is perfect. 
Berger et al. \cite{berger2010signal} employed multiple signal classifiers and compressed sensing techniques to obtain better target resolutions and better clutter removal performance. Taking into account the demodulation error, Zheng \& Wang 
\cite{zheng2017super} proposed a delay and Doppler shift estimation method which uses the received OFDM signal from an un-coordinated but synchronized illuminator.

In OFDM passive radar system, it was known that demodulation can be implemented by using a reference signal \cite{colone2009multistage}. Since the demodulation can achieve better accuracy than the direct use of a reference signal, a more accurate matched filter can be implemented on the basis of the data symbols estimated in advance. Thereby, the performance of passive radar is greatly improved.
The exact matched filter formulation for passive radar using OFDM waveforms was derived by Berger et al. \cite{berger2010signal},
in which the signal-to-noise ratio (SNR) is improved by reducing the noise’s spectral bandwidth to that of the wavelet.
Later, Zheng et al. \cite{zheng2017super} developed an explicit model for an accurately matched filter.
Suppose there are $K$ paths in the surveillance area. The delay and Doppler frequency are denoted as $\tau_k$ and $f_k$, respectively. Then, the relationship between the received signal $y(t)$ and the transmitted signal $x(t)$ takes the following form
\begin{equation}\label{eq:y}
	y(t) = \sum_{k=1}^K A_k e^{i2\pi f_k t} x(t-\tau_k) + w(t),
\end{equation}
where $w(t)$ is an additive noise, and $A_k$ is a complex coefficient to characterize the attenuations such as path loss,
reflection, and processing gains. 
When $\tau_k$ and $f_k$ are achieved, the range and velocity of the targets can be determined subsequently. 
In order to obtain   useful information, it should process the received signal for time-frequency conversion by utilizing FFT.
Mathematically, to explore the signal's structure, the received signal taken by FFT can be conveniently denoted by matrices, and then be reordered as a vector. More precisely, if the received signal is denoted as $\bm r$, then, \eqref{eq:y}  can be reshaped as the following vector form
\begin{equation}\label{eq:rbm}
	\bm{r} = \bm{\tilde S}\bm{z} + \bm{e} + \bm{v} 
	= \bm{\tilde S} \left(\bm G(\bm \psi)^C \circ \bm B(\bm \phi) \right) \bm{\alpha} + \bm{e} + \bm{v},
\end{equation}
where $\bm{\tilde S}$ is a diagonal matrix related to data symbols; $\bm \psi$, $\bm \phi$, and $\bm \alpha$ are unknown vectors associated with the delay, Doppler frequency, and attenuations, respectively; $\bm e$ is a demodulation error vector and $\bm v$ is an error. 
Besides, the symbol ``$\circ$'' denotes Khatri-Rao product, and ``$(\cdot)^C$'' represents conjugate. 
Using these notations, we see that estimating $\tau_k$ and $f_k$ turns to 
jointly estimate the parameters $\bm \phi$,  $\bm \psi$, and $\bm \alpha$.
Noting that the 2nd-equality  \eqref{eq:rbm} is nonlinear, which may take great effort to estimate $\bm \phi$,  $\bm \psi$, and $\bm \alpha$ together.
Meanwhile, the sparsity of signals should be taken into account because  the targets and clutters are indeed sparsely distributed in space, so do the reflected signals. Besides, the demodulation error rate of a communication system is typically low under  normal operating conditions, so the demodulation error signal should be sparse, too.

%介绍频谱估计算法。。。。
We note that the spectrum-based parameter estimation methods can be roughly divided into two categories, namely, the  beamforming techniques and the subspace-based approaches. 
One of the most typical subspace-based methods is named multiple signal classification (MUSIC) which has higher superresolution, see Kim et al. \cite{krim1996two}. 
In recent years, the employment of compressive sensing (CS) methodologies into the problem of signals' estimation has been attracted tremendous attention. 
For instance, Berger et al. \cite{berger2010signal} implemented  MUSIC as a two-dimensional spectral estimator by using  spatial smoothing, and then used CS to identify the targets. But this approach needs an additional step to remove the dominant clutter and direct signals. 
Subsequently, Zheng  et al. \cite{zheng2017super} applied the CS technique to simplify the joint estimation problem (\ref{eq:rbm}) as a convex optimization problem.
It was shown in \cite{candes2013super, candes2014towards} that, the frequencies can be exactly recovered via convex optimization once the separation between the frequencies is larger than a certain threshold. 
Later, this result was extended to a continuous-frequency-recovery problem based on an atomic-norm regularized minimization \cite{tang2013compressed, bhaskar2013atomic}. Moreover, Bhaskar et al. \cite{bhaskar2013atomic} showed that the approach based on  a semi-definite program (SDP) is superior to that based on an $\ell_1$-norm although  the latter  is generally better than some classical line spectral analysis approaches.

As we know, the CS methodologies concentrate on  discrete dictionary coefficients, but the signals from the moving targets are specified by some parameters in a continuous domain. 
In order to apply CS into a continuous scenario, it  usually adopts a discretization procedure to reduce the continuous parameter space to a finite set of grid points, see e.g., \cite{herman2009high,bajwa2010compressed,hu2012compressed}. 
These simple strategies might yield good performance for problems where the true parameters lie on the grids, but there is also a drawback of mismatch \cite{yang2012robustly,yang2012off}. 
To resolve this dilemma, Tang  et al. \cite{yang2016exact} employed an atomic-norm to work directly on the continuous frequency  instead of on a grid, that is,
exactly identifying the unknown frequencies on any values in a normalized frequency domain.
Utilizing the atomic-norm induced by samples of complex exponentials, it is appropriate to estimate $\bm z$ and $\bm e$ through the following atomic-norm regularized minimization
\begin{eqnarray}
	\label{eq:mod0}
	(\bm{\bar z},\bm{\bar e}) = \mathop {\arg \min }\limits_{\bm{z}\in \mathbb{C}^{MN},\bm{e}\in \mathbb{C}^{MN}} \frac{1}{2}{\left\| {{\bm{r}} - \bm{e} - \bm{\tilde S}\bm{z}} \right\|_2^{2}} + \lambda \| \bm{z} \|_{\cal A} + \mu \| \bm{e} \|_1, 
\end{eqnarray}
where $\lambda > 0$ and $\mu>0$ are weighting parameters, $\| \cdot\|_{\cal A}$ is a so-called atomic-norm induced by an atom set $\cal A$, $\bm r\in \mathbb{C}^{MN}$ is a received signal, and $\bm{\tilde S}$ is a $MN\times MN$ matrix related to data symbols. 
When $\bm{z}$ is achieved, the delay $\tau_k$ and Doppler frequency  $f_k$ appeared in \eqref{eq:y} can be obtained subsequently. 
From Yang et al. \cite{yang2016vandermonde}, it is shown that the model \eqref{eq:mod0} can be reformulated as the following SDP problem
	\begin{eqnarray}\label{eq:mod1}
	\begin{aligned}
		&\min\limits_{\bm{e},\bm{z},\epsilon,	\bm{U} 	}\quad \frac{1}{2}\|\bm{r}-\bm{e}-\tilde{S}\bm{z}\|^2_2 + \frac{\lambda}{2MN}{\rm Tr}\big(\mathcal T(\bm{U})\big) + \frac{\lambda \epsilon}{2} + \mu\| \bm{e}\|_{1}
		\\
		&\qquad \text{s.t.} \qquad \begin{pmatrix}
			\mathcal T(\bm{U}) & \bm{z} \\
			\bm{z}^H  & \epsilon
		\end{pmatrix} \succeq 0,
	\end{aligned}
\end{eqnarray}
where $\epsilon>0$ is an unknown variable, $\mathcal T(\bm{U})$ is a $MN\times MN$ block Toeplitz matrix related to $\bm U\in \mathbb{C}^{(2M-1)\times(2N-1)}$,  ``$(\cdot)^H$'' represents conjugate transpose, and the symbol `$\succeq$' indicates that the matrix is positive and semi-definite.

In light of the above analysis, we see that the estimation can be implemented by developing an efficient algorithm to find an optimal solution of \eqref{eq:mod1}. Zheng et al. \cite{zheng2017super} directly employed an alternating direction method of multipliers (ADMM)  on \eqref{eq:mod1} which has been displayed better performance according to some simulation results. We note that the variables in ADMM are computed individually in a Gauss-Seidel manner, and all the subproblems are solved efficiently.
But, the convergence of the ADMM to  \eqref{eq:mod1} is lack of certificate, which may lead to the algorithm failing occasionally. For a counter-example, one may refer to the  paper of Chen et al.  \cite{chen2016direct}. 
It is because of this, an algorithm which is at least as efficient as the directly-extended ADMM but with  a convergence certificate is very necessary. 
To achieve this goal, in this paper, we recast the variables $(\bm{e},\bm{z},\epsilon,	\bm{U})$ as two groups and then compute the subproblems at each group in a symmetric Gauss-Seidel (sGS) order. The sGS used in each group  needs to compute some subproblems twice but it can improve the performance evidently because it has the ability to ensure convergence. 
It should be emphasized that the sGS was firstly developed by Li et al. \cite{li2016schur}, and was successfully used to solve many multi-block conic programming problems (e.g.,  \cite{chen2017efficient,li2018efficient, sun2015convergent}) and some image
processing problems (e.g., \cite{ding2018symmetric}). 
The great feather of the sGS based ADMM  is that 
it inherits several advantages of ADMM and is capable of solving some large multi-block composite problems.
Most importantly, the sGS decomposition theorem of Li et al. \cite{li2019block} established the equivalence between the sGS based ADMM and a traditional ADMM, so that its convergence can be followed directly from Fazel et al. \cite{fazel2013hankel}. 
We do  numerical experiments using some simulated data which illustrates that the sGS based ADMM always  detects the moving targets clearly and it is also superior to  ADMM.

The remainder of the paper is organized as follows. In section \ref{sec2}, we quickly describe  the passive radar signal system and formally review the joint delay-Doppler estimation model. Some basic concepts and preliminary knowledge of optimization are also included. Section \ref{sec3} is devoted to applying the sGS based ADMM to solve problem \eqref{eq:mod1}, subsequently, the approach of recovering the frequency and amplitude is also described. Section \ref{sec4} provides some simulated experiments to show the numerical superiority of sGS-ADMM.

\section{Preliminaries}\label{sec2}
In this section, we briefly review the exact matched filter formulation in passive radar using OFDM waveforms and display some details on   model \eqref{eq:rbm}. 
In addition, we also introduce some basic concepts and definitions in optimization used frequently in subsequent developments.

\subsection{Descriptions on passive radar system }\label{sub2.1}
At the beginning of this part, we briefly give a blanket of settings on passive radar systems, and then develop an accurately matched filtering formula. The settings include:
(i) the received OFDM signal comes from an un-coordinated but synchronized illuminator;
(ii) the passive radar system only performs  demodulation on data symbols, but not on forward error correction decoding.

The OFDM is a multicarrier modulation scheme which can be described as follows: suppose that there are $M$ blocks in transmitted data and there are $N$ orthogonal subcarriers in each block.
Accordingly, the data symbols in $m$-th block and $n$-th subcarrier are denoted by $s_m(n)$. 
As the demodulation type is usually known, then the data-symbol $s_m(n)$ can be estimated and  the matched filter can be implemented. 
Given a rectangular window of length $T$ at the receiver, the frequency spacing is denoted as $\bigtriangleup f := 1/T$, and the duration of each transmission block is denoted as $\bar T := T+T_\text{cp}$, where $T_\text{cp}$ is the length of a cyclic prefix. 
Using these notations, the baseband   signal over the $M$ blocks can be expressed as
\begin{align*}
	\label{eq:x}
	x(t):=\sum_{m=0}^{M-1}x_m(t)
	:=\sum_{m=0}^{M-1} \sum_{n=0}^{N-1} s_m(n) e^{i 2\pi n \Delta f t } \xi(t - m \bar T),
\end{align*}
where $x_m(t)$ is the baseband signal  in the $m$-th block with $ m \bar T - T_\text{cp} \leq t \leq m \bar T + T$,
 and $\xi(t) =1$ if $t \in [-T_\text{cp},T]$ and $0$ otherwise.
Indexing the return of the $k$-th arrival and its associated Doppler shift and delay, the received signal  and the  received signal $y(t)$ takes the following form:
\begin{equation}
	\label{eq:y2}
	y(t) := \sum_{k=1}^K A_k e^{i2\pi f_k t} x(t-\tau_k) + w(t),
\end{equation}
where $K$ is the number of paths, $w(t)$ is an additive noise, $A_k$ is an attenuations coefficient, $\tau_k$ is a delay, and $f_k$ is a  Doppler frequency. 

If taking Fourier transform on $y(t)$ in $m$-th block, the resulting signal in $n$-th subcarrier is formulated as:
\begin{equation}\label{eq:rm1}
	\begin{split}
		r_m(n) :=& \int_{m \bar T}^{m \bar T + T} e^{-i 2\pi n \Delta f t} y(t)dt\\
		=& \sum_{k=1}^{K}\sum_{m'=0}^{M-1}\sum_{l=0}^{N-1}A_k s_{m'}(l)\int_{m \bar T}^{m \bar T + T} e^{i 2\pi (f_k t+(l- n) \Delta f t-l \Delta f \tau_k)} \xi(t -\tau_k- m' \bar T) dt \\
		&+ \int_{m \bar T}^{m \bar T + T} e^{-i 2\pi n \Delta f t} w(t)dt.
	\end{split}
\end{equation}
Assuming that the largest possible delay is smaller than the cyclic prefix, then the term $\xi(t-\tau_k-m'\bar T)$ is $1$ in the integration interval  $[m\bar{T}, m\bar{T}+T]$ and $0$ otherwise. Usually, the integration time is almost on the order of a second, which means that when the product between $\bar{T}$ and the Doppler frequency is small compared to unity, we can approximate the phase rotation within one OFDM block as constant \cite{berger2010signal}, that is,
$$
e^{i2\pi f_k t} \approx e^{i2\pi f_k m \bar T},\quad \forall t\in [m\bar{T}, m\bar{T}+T].
$$

For notational simplicity, we denote $\alpha_k := A_k T $, $\phi_k := f_k \bar T \in [0,1)$ and $\psi_k := \Delta f \tau_k \in [0,1)$.
Besides, we note that  $\int_{m \bar T}^{m \bar T + T} e^{i 2\pi (l-n) \Delta f t}dt=0$ if $l\neq n$.  Then $r_m(n)$ can further be rewritten as
\begin{equation}
	\begin{split}
		\label{eq:rm}
		r_m(n)=& s_m(n) \sum_{k=1}^{K} A_k e^{i 2\pi m f_k \bar T} \int_{m \bar T}^{m \bar T + T} e^{- i 2\pi n \Delta f  \tau_k } dt + \int_{m \bar T}^{m \bar T + T} e^{-i 2\pi n \Delta f t} w(t)dt\\
		=& s_m(n) \sum_{k=1}^{K} A_k \bar{T}e^{i 2\pi m f_k \bar T}e^{- i 2\pi n \Delta f  \tau_k } dt + \int_{m \bar T}^{m \bar T + T} e^{-i 2\pi n \Delta f t} w(t)dt\\
		=&s_m(n)  \sum_{k=1}^{K} \alpha_k e^{i(2\pi m \phi_k - 2\pi n \psi_k)}+ \int_{m \bar T}^{m \bar T + T} e^{-i 2\pi n \Delta f t} w(t)dt\\
		=& s_m(n)z_m(n)   + v_m(n),
	\end{split}
\end{equation}
 where
 \begin{equation}\label{eq:z}
 	z_m(n) := \sum_{k=1}^{K} \alpha_k e^{i(2\pi m \phi_k - 2\pi n \psi_k)}\quad\text{and}\quad
    v_m (n):= \int_{m \bar T}^{m \bar T + T} e^{-i 2\pi n \Delta f t} w(t) dt,
\end{equation}
in which, $v_m (n)$ is assumed to be a complex Gaussian variable with zero mean and variance $\sigma^2$. 
For convenience, we denote the estimated data-symbol  in $m$-th block as $\hat{s}_m(n)$. Then the demodulation error, denoted by $e_m(n)$, takes the following form
\begin{equation}\label{eq:e}
		e_m(n) := r_m(n) - \hat s_m(n) z_m(n) - v_m(n).
\end{equation}
It should be noted that the demodulation error rate of a communication system is typically low under some normal operating conditions. Hence, most of  $e_m(n)$ are zero and the noise caused by  demodulation error is always sparse.

Let $\bm{\hat S} \in \mathbb{C}^{M \times N}$, $\bm R \in \mathbb{C}^{M \times N}$, $\bm{Z} \in \mathbb{C}^{M \times N}$, $\bm E \in \mathbb{C}^{M \times N}$ and $\bm V \in \mathbb{C}^{M \times N}$ be matrices whose $(m,n)$-th element are denoted as $\hat s_m(n)$, $r_m(n)$, $z_m(n)$, $e_m(n)$ and $v_m(n)$, respectively. Let $\bm \phi := (\phi_1, \phi_2,...,\phi_K)^{\top}$ and $\bm \psi := (\psi_1, \psi_2,...,\psi_K)^{\top}$.  Accordingly, the response matrices are defined as $\bm B(\bm \phi):= (\bm b(\phi_1), \bm b(\phi_2),..., \bm b(\phi_K))\in\mathbb{C}^{M\times K}$ and $\bm G (\bm \psi) :=( \bm g(\psi_1),\bm g(\psi_2),...,\bm g(\psi_K))\in\mathbb{C}^{N\times K}$, in which, $\bm b(\phi) := (1, e^{i 2 \pi \phi},\ldots,e^{i 2 \pi (M-1) \phi})^{\top}\in\mathbb{C}^M$ and $\bm g(\psi) := (1, e^{i 2 \pi \psi},\dots,e^{i 2 \pi (N-1) \psi})^{\top}\in\mathbb{C}^N$. 
Using these notations, the relation \eqref{eq:e} can be rewritten in a matrix form, that is,
\begin{equation}\label{eq:R}
		\bm R = \bm{\hat S} \odot \bm{Z} + \bm E + \bm V 
		= \bm{\hat S} \odot \left( {\bm B}(\bm \phi) \text{diag}(\bm{\alpha}) \bm{G}(\bm \psi)^C \right) + \bm E + \bm V,
\end{equation}
where $\bm{\alpha}:= (\alpha_1,\alpha_2,..., \alpha_K)^{\top}\in\mathbb{C}^{K}$ and ``$\odot$" is a Hadamard product.  Utilizing the structure of the signal, the matrices in \eqref{eq:R}  can be vectorized  to obtain a more concise and intuitive representation, that is,
\begin{equation*}
	\bm{r} = \bm{\tilde S} \left(\bm G(\bm \psi)^C \circ \bm B(\bm \phi) \right) \bm{\alpha} + \bm{e} + \bm{v},
\end{equation*}
where $\bm{\tilde S} = \text{diag}(\text{vec}(\bm{\hat S})) \in \mathbb{C}^{MN \times MN}$, $\bm{r} = \text{vec}(\bm{R}) \in \mathbb{C}^{MN}$, $\bm{z} = \text{vec}(\bm{Z}) \in \mathbb{C}^{MN}$, $\bm{e} = \text{vec}(\bm E) \in \mathbb{C}^{MN}$ and $\bm{v} = \text{vec}(\bm V) \in \mathbb{C}^{MN}$. In this formula, the term $\left(\bm G(\bm \psi)^C \circ \bm B(\bm \phi) \right) \in \mathbb{C}^{MN \times K}$ is a matrix whose $k$-th column has the form $\bm g(\psi_k)^C \otimes \bm b(\phi_k)$ in which $\otimes$ denotes Kronecker product.

\subsection{Joint delay-Doppler estimation model}\label{sub2.2}

From the work of Chandrasekaran et al. \cite{chandrasekaran2012convex} that, the atomic-norm  is used to  characterize  a sparse combination of sinusoids.
Hence, define an atom as $\bm a(\phi, \psi) = \bm g(\psi)\otimes \bm b(\phi)$ where `$\otimes$' denotes Kronecker product. Obviously, we have that $\bm{z} := \sum_{k=1}^K \alpha_k \bm a(\phi_k,\psi_k)$. The atomic-norm can enforce sparsity in the atom set in the form of ${\cal A} := \{ \left.\bm a(\phi,\psi) \right| \phi \in [0,1), \psi \in [0,1) \}$ when it has been normalized.
The definition of atomic-norm is given as follows: 
\begin{definition}\label{def2.1}
 	The 2D atomic-norm for $\bm{z} \in \mathbb{C}^{MN}$ is defined as
	\begin{equation*}
		\| \bm{z} \|_{\cal A} := \inf \left\{\left.\epsilon>0 ~\right| {\bm z}\in \epsilon\,\rm{conv}\mathcal A\right\}
		= \mathop {\inf }\limits_{\scriptstyle{\phi_k} \in [0,1), \scriptstyle{\psi_k} \in [0,1), 
			\scriptstyle{\alpha_k} \in \mathbb{C}\hfill} \left\{ \left. \sum\limits_{k=1}^K {|\alpha_k| } ~\right| \bm{z} = \sum_{k=1}^K  {\alpha_k\bm a(\phi_k,\psi_k)} \right\}.
	\end{equation*}
\end{definition}
\noindent   
From the work of Yang et al. \cite{yang2016vandermonde}, it  shows that the atomic-norm is closely related to a SDP problem containing a 
Toeplitz matrix,  that is,
\begin{eqnarray}
	\label{eq:atomic}
	\| \bm{z} \|_{\cal A} = \mathop {\inf }\limits_{\bm U, \epsilon} \left\{ \left. 
		\frac{1}{2MN}{\rm Tr}({\cal T}(\bm U)) + \frac{\epsilon }{2} ~\right|
		\text{s.t.} ~ {\begin{pmatrix}
				{{\cal T}(\bm U)} & \bm{z} \\
				{{\bm{z}^H}}      & \epsilon
	 \end{pmatrix}}  \succeq 0
\right\},
\end{eqnarray}
where ${\rm Tr}(\cdot)$ is the trace of a matrix, and $\bm U$ is a $(2M-1) \times (2N-1)$ matrix in the form of
\begin{eqnarray}
	\label{eq:U}
	\bm U = (\bm u_{-N+1},\bm u_{-N+2},\ldots,\bm u_0,\ldots,\bm u_{N-1}),
\end{eqnarray}
where $\bm u_l$ is a column vector with the form of $\bm u_l:=(u_l(-M+1),u_l(-M+2),\ldots,u_l(M-1))^{\top}$, and $\mathcal T(\bm{U})$ is a $MN\times MN$ block Toeplitz matrix in the form of
\begin{eqnarray}
	\label{eq:calT}
	{\cal T}(\bm U) =  \begin{pmatrix}
			{\text{Toep}(\bm u_0)}      &{\text{Toep}(\bm u_{-1})}  &{\cdots} &{\text{Toep}(\bm u_{-N+1})}\\
			{\text{Toep}(\bm u_1)}      &{\text{Toep}(\bm u_0)}     &{\cdots} &{\text{Toep}(\bm u_{-N+2})}\\
			{\vdots}                    &{\vdots}                   &{\ddots} &{\vdots}                  \\
			{\text{Toep}(\bm u_{N-1})}  &{\text{Toep}(\bm u_{N-2})} &{\cdots} &{\text{Toep}(\bm u_0)}
	\end{pmatrix},
\end{eqnarray}
with
\begin{eqnarray}
	\label{eq:T}
	\text{Toep}(\bm u_l) = {\begin{pmatrix}
			{u_{l}(0)}  & {{u_{l}(-1)}}   & {\cdots}   & {{u_{l}(-M+1)}}  \\
			{u_{l}(1)}  &  {u_{l}(0)}     & {\cdots}   & {u_{l}(-M+2)}    \\
			{\vdots}    &   {\vdots}      & {\ddots}   &   {\vdots}       \\
			{u_{l}(M-1)}& {u_{l}(M-2)}    & {\cdots}   & {u_{l}(0)}
	\end{pmatrix}}.
\end{eqnarray}

At the end of this section, we quickly review the definition of proximal point operator defined in real space \cite{moreau1962fonctions} which is used frequently at the subsequent analysis.

\begin{definition}\label{def2.2}
	Let $f(x) :\mathbb{R}^n\to\mathbb{R}\cup\left\{+\infty\right\}$ be a closed proper convex function. The proximal point mapping of $f$ at $y$, denoted by $\text{Prox}_f^{\mu}:\mathbb{R}^n\to\mathbb{R}^n$, takes the following from
	\begin{displaymath}
		\text{Prox}_f^{\mu}(y)=\argmin_{x\in\mathbb{R}^n}\left\{f(x)+\frac{1}{2\mu}\|x-y\|_2^2\right\},
	\end{displaymath}
where $\mu>0$ is a given scalar. The symbol $\text{Prox}_f^{\mu}(y)$ is simply written as $\text{Prox}_f(y)$ in the special case of $\mu = 1$.
\end{definition}
%play a cruil role

It is known that $\text{Prox}_f^{\mu}(y)$ is properly defined for every $y\in \mathbb{R}^n$ if 
$f$ is  convex and proper. Specially, if $f(x)$ is an $\ell_1$-norm function,  it can be expressed explicitly as $\text{Prox}_f^{\mu}(y)=\text{sign}(y)\odot \max\left\{|y|-\mu , 0\right\} $. 
If $f(x)$ is an indicator function over a closed set $\mathcal{K}$, denoted by $\delta_{\mathcal{K}}(x)$, i.e., $0$ if $x\in \mathcal{K}$ and $\infty$ otherwise, then the proximal point mapping $\text{Prox}_{\delta_{\mathcal{K}}(\cdot)}^{\mu}(y)$ is equivalent to an orthogonal  projection over $\K$, i.e., $\Pi_{\mathcal{K}}(y)$.

%%%%%%%%%%%%%%%%%%%%%%%%%%%%%%%%%%%%%%%%%%%%%%
\section{Estimation model and optimization algorithms}\label{sec3}
%%%%%%%%%%%%%%%%%%%%%%%%%%%%%%%%%%%%%%%%%
In the remaining part of this paper, our discussion is based on the complex Hermitian space with inner product $\langle \cdot,\cdot\rangle_{\mathbb{R}}$, i.e., $\langle \bm x,\bm y\rangle_{\mathbb{R}}=\text{Re}(\bm x^H\bm y)$ for arbitrary  vectors $\bm x\in \mathbb{C}^n$ and $\bm y\in \mathbb{C}^n$.

\subsection{Estimation model and directly-extended ADMM}
For convenience, we let $\bm\Theta := \begin{pmatrix}
	\mathcal T(\bm{U}) & \bm{z} \\
	\bm{z}^H  & \epsilon
\end{pmatrix}\in \mathbb{C}^{(MN+1)\times(MN+1)}$, then  model \eqref{eq:mod1} is reformulated as
	\begin{eqnarray}\label{eq:mod2}
	\begin{aligned}
		&\min\limits_{\bm{e},\bm{z},\epsilon,\bm{\Theta}, 	\bm{U} 	}\quad \frac{1}{2}\|\bm{r}-\bm{e}-\tilde{S}\bm{z}\|^2_2 + \frac{\lambda}{2MN}{\rm Tr}\big(\mathcal T(\bm{U})\big) + \frac{\lambda \epsilon}{2} + \mu\| \bm{e}\|_{1} +\delta_{\mathcal{S}^{MN+1}_+}(\bm \Theta)
		\\[2mm]
		&\qquad \text{s.t.} \qquad \bm\Theta =\begin{pmatrix}
			\mathcal T(\bm{U}) & \bm{z} \\
			\bm{z}^H  & \epsilon
		\end{pmatrix},
	\end{aligned}
\end{eqnarray}
where  ${\bm z}\in \mathbb{C}^{MN}$, ${\bm e}\in \mathbb{C}^{MN}$, $\bm{U}\in \mathbb{C}^{(2M-1)\times (2N-1)}$, $\epsilon\in \mathbb{C}$, the notation $\S_+^{MN+1}$ represents a symmetric and semi-positive matrix space with dimension $MN+1$.
From the work of Tang et al. \cite{tang2013compressed} that, the dual  of \eqref{eq:mod2} takes the following form
\begin{eqnarray}\label{eq:dual}
	\max_{\bm \nu\in \mathbb{C}^{MN}} &{\left\langle {{(\bm{\tilde S}^H)^{-1} \bm \nu},\bm{r}} \right\rangle _{\mathbb R}} - \frac{1}{2}\left\| {(\bm{\tilde S}^H)^{-1} \bm \nu} \right\|_2^2 \\[2mm]
	\text{s.t.} & \|\bm \nu \|_{\cal A}^H \leq \lambda, \quad \left\| (\bm{\tilde S}^H)^{-1} \bm \nu \right\|_\infty \leq \mu\nonumber,
\end{eqnarray}
where $\bm \nu \in \mathbb{C}^{MN}$ is a dual variable, and $\| \bm \nu \|_{\cal A}^H = \sup_{\| \bm z \|_{\cal A} \leq 1} \langle \bm \nu , \bm z \rangle_{\mathbb R}$ is called dual norm of $\|\cdot\|_{\A}$. 

Let $\widetilde{\mathcal{L}}_{\rho}({\bm{e}},{\bm{z}},\epsilon,\bm U,\bm{\Theta}; \bm{\Gamma})$ be an augmented Lagrangian function associated with \eqref{eq:mod2},
where $\bm \Gamma\in\mathbb{C}^{(MN+1)\times(MN+1)}$ is a multiplier.
The directly-extended ADMM employed by Zheng  et al.  \cite{zheng2017super}  minimizes $\widetilde{\mathcal{L}}_{\rho}({\bm{z}}, \epsilon, \bm U, \bm{e}, \bm{\Theta}; \bm{\Gamma})$ in a simple Gauss-Seidel manner with respect to each variable by ignoring the relationship between them. More explicitly, with given $(\bm{z}^k, \epsilon^{k}, \bm{U}^{k}, \bm{e}^{k}, \bm{\Theta}^{k})$, the next point $({\bm{z}}^{k+1}, \epsilon^{k+1} ,\bm{U}^{k+1}, \bm{e}^{k+1}, \bm{\Theta}^{k+1})$ is generated in the order of $\bm{z}\to \epsilon\to\bm U\to\bm{e}\to\bm{\Theta}$.
As shown by Zheng et al. \cite{zheng2017super}  that, each step  involves solving a convex minimization problem which admits closed-form solutions by taking its favorable structures. 
Even though better performance has been achieved experimentally, the convergence of ADMM is still lack of certificate. For more details on this argument, one may refer to a counter-example of Chen et al. \cite{chen2016direct}. 
Because of this, it is ideal to design a convergent variant method which is at least as efficient as  ADMM.

\subsection{sGS based ADMM}
This part is devoted to designing a sGS based ADMM for solving \eqref{eq:mod2}. For convenience, we let $\bm{g} := {\bm{r}} - {\bm{e}} - \tilde{S}{\bm{z}}$, then \eqref{eq:mod2} is rewritten equivalently as the following from
	\begin{equation}\label{eq:mod3}
		\begin{split}
			&\argmin\limits_{\bm{e},\bm{g},\bm{z}, \epsilon,\bm{U},\bm{\Theta}} \quad \frac{1}{2}\|\bm{g}\|^{2}_{2} + \frac{\lambda}{2MN}\rm {Tr}(\mathcal T(\bm{U})) + \frac{\lambda \epsilon}{2} + \mu\|{\bm{e}}\|_{1} + \delta_{\mathcal{S}^{MN+1}_+}(\bm \Theta)
			\\
			&\qquad \text{s.t.} \qquad	\quad\bm{g} = \bm{r} - \bm{e} - \tilde{S}\bm{z},
			\\
			&\qquad \qquad \qquad \bm{\Theta} = 	\begin{pmatrix}
				\mathcal T(\bm{U}) & \bm{z} \\
				\bm{z}^H  & \epsilon
			\end{pmatrix}.
		\end{split}
	\end{equation}	
We note that although problem \eqref{eq:mod3} has separable structures with regarding to objective function and constraints, it actually contains six blocks, which indicates that the traditional two-block ADMM will no longer be used.
To address this issue, we  partition these variables into two groups, for example, we can view $(\bm e,\bm g)$ as one group
and $(\bm{z}, \epsilon,\bm{U},\bm{\Theta})$ as another. For clearly catching this partition, we let
\begin{eqnarray}
	\bm{Y}:=\begin{pmatrix}
	\begin{pmatrix}
		\bm{0}_{MN}	& \bm e\\
		\bm{0}^{\top}	    &   0
	\end{pmatrix}	
	\\
	\begin{pmatrix}
		\bm{0}_{MN}	& \bm{g} \\
		\bm{0}^{\top}       &    0
	\end{pmatrix}
\end{pmatrix}	
\triangleq  \begin{pmatrix}
	\bm{Y}_1	    \\
	\bm{Y}_2
\end{pmatrix} ,
\qquad	
\bm{W}:=\begin{pmatrix}
	\bm{\Theta}	\\
	\begin{pmatrix}
		\bm{0}_{MN}    	&   \bm z  \\
		{\bm z}^H	&    0
	\end{pmatrix} 
	\\
	\begin{pmatrix}
		\mathcal{T}(\bm{U})	&   \bm 0 \\
		\bm{0}^{\top}	                &   \epsilon
	\end{pmatrix} 
\end{pmatrix},
\end{eqnarray}
where `$\bm 0$' denotes a zero vector, and `$\bm 0_{MN}$' denotes a zero matrix with order $MN$.
Using both notations,   model \eqref{eq:mod3} is transformed into the following form with respect to two variables $\bm Y$ and $\bm W$, that is,
\begin{equation}\label{eq:blockmod}
	\begin{split}
		\argmin_{\bm{Y},\bm W} &\quad  
		\mu\|\bm{Y}_1\|_{1} + \frac{1}{2}\| \left( \bm{0}_{MN+1}~,~\bm{I}_{MN+1}  \right) \bm{Y}        
		\|^{2}_{F} +\delta_{\mathcal{S}^{MN+1}_+}(\bm \Theta)\\
		&\qquad\qquad\qquad\qquad\qquad\qquad\qquad+ \frac{\lambda}{2}   {\rm Tr}  \left( \begin{pmatrix}
			\bm{0}_{MN+1}, ~  \bm 0_{MN+1},~
			\begin{pmatrix}  \frac{1}{MN}\bm{I}_{MN}   & \bm{0}	\\ 
				\bm{0}^{\top}           & 1
			\end{pmatrix}
		\end{pmatrix}\bm W \right)
		\\[2mm]   
		\text{s.t.}  & \quad 
		\begin{pmatrix}
			\bm{0}_{MN+1} & \bm{0}_{MN+1}\\ 
			\bm{I}_{MN+1} & \bm{I}_{MN+1}
		\end{pmatrix} \bm{Y} 
		+   
		\begin{pmatrix}
			-\bm{I}_{MN+1}  & \bm{I}_{MN+1} & \bm{I}_{MN+1}\\
			\bm{0}_{MN+1}   &\begin{pmatrix}
				\widetilde{\bm{S}} & \bm{0}\\
				\bm{0}^{\top}       &      0
			\end{pmatrix}           &  \bm{0}_{MN+1}
		\end{pmatrix}   \bm W 
		=    
		\begin{pmatrix}
			\bm{0}_{MN+1}  \\  \begin{pmatrix}
				\bm{0}_{MN} & \bm{r}\\
				\bm{0}      &   0
			\end{pmatrix}
		\end{pmatrix},
	\end{split}
\end{equation}
where   $\bm I$ is an identity matrix with an appropriate size.
Obviously, we see that there are two nonsmooth terms in the objective function of \eqref{eq:blockmod}, one is about $\bm Y_1$ in group $\bm Y$ and the other one is about $\bm\Theta$ in   group $\bm W$, and each group possesses a ``nonsmooth+smooth" structure. 

Let $\rho>0$ be a penalty parameter and $\mathcal{L}_{\rho}({\bm{e}},\bm{g},{\bm{z}}, \epsilon,\bm U,\bm{\Theta} ;\bm{\beta} , \bm{\Gamma})$ be an augmented Lagrangian function associated with \eqref{eq:mod3}, that is,
\begin{eqnarray}\label{eq:lag}
	\mathcal{L}_{\rho}({\bm{e}},\bm{g},{\bm{z}}, \epsilon,\bm U,\bm{\Theta} ;\bm{\beta} , \bm{\Gamma}):=&
	\frac{1}{2}\|\bm{g}\|^{2}_{2} + \frac{\lambda}{2MN}{\rm Tr}(\mathcal T(\bm{U})) + \frac{\lambda \epsilon}{2} + \mu\|{\bm{e}}\|_{1} + \delta_{\mathcal{S}^{MN+1}_+}(\bm \Theta) \nonumber\\
	&+ \langle \bm{\beta},\bm{g} - {\bm{r}} + {\bm{e}} + \tilde{S}{\bm{z}} \rangle_{\mathbb{R}}
	+\left\langle \bm{\Gamma} , \bm{\Theta} -\begin{pmatrix}
		\mathcal T(\bm{U}) & {\bm{z}} \\
		{\bm{z}}^H         &  \epsilon
	\end{pmatrix}  \right\rangle_{\mathbb{R}}          
	\nonumber\\
	&+	\frac{\rho}{2}\| \bm{g} - {\bm{r}} + {\bm{e}} + \tilde{S}{\bm{z}} \|^2_2
	+ \frac{\rho}{2}\left\| \bm{\Theta} -	\begin{pmatrix}
		\mathcal T(\bm{U}) & {\bm{z}} \\
		{\bm{z}}^H         & \epsilon
	\end{pmatrix}    \right\|^2_F,
\end{eqnarray}
where $\bm{\beta}\in\mathbb{C}^{MN\times 1}$ and $\bm \Gamma\in\mathbb{C}^{(MN+1)\times(MN+1)}$ are multipliers. The sGS based ADMM    minimizes $\mathcal{L}_{\rho}({\bm{e}},\bm{g},{\bm{z}}, \epsilon,\bm U,\bm{\Theta} ;\bm{\beta} , \bm{\Gamma})$ alternatively with respect to two groups $(\bm e, \bm g)$ and $(\bm z,   \epsilon,\bm U,\bm\Theta)$, and then adopts sGS technique in each group. 
For example, in the $(\bm e, \bm g)$-group, it takes the order of  $\bm g\to \bm{e}\to \bm g$ in the sense of computing $\bm g$ twice, while in the  $(\bm z,   \epsilon,\bm U,\bm\Theta)$-group, it  takes the order of $\bm z\to \epsilon\to\bm U\to \bm\Theta \to \bm U\to \epsilon\to\bm z$, i.e., computing $(\bm z,   \epsilon,\bm U)$ twice. 
In light of these analysis, we get that, with given $(\bm e^k, \bm g^k)$, the next 
$(\bm e^{k+1}, \bm g^{k+1})$ is obtained using the following scheme 
\begin{equation}\label{geg}
\begin{cases}
	&\widetilde{\bm{g}}^k :=  \argmin\limits_{\bm{g}}	\mathcal{L}_{\rho}(\bm{e}^k,\bm{g},\bm{z}^k, \epsilon^k,\bm{U}^k,\bm{\Theta}^k;\bm{\beta}^k , \bm{\Gamma}^k),\\
	&\bm{e}^{k+1} := \argmin\limits_{\bm{e}}	\mathcal{L}_{\rho}(\bm{e}, \widetilde{\bm{g}}^k, \bm{z}^k, \epsilon^k, \bm{U}^k,\bm{\Theta}^k;\bm{\beta}^k , \bm{\Gamma}^k),\\
	&{\bm{g}}^{k+1}  := \argmin\limits_{\bm{g}}	\mathcal{L}_{\rho}(\bm{e}^{k+1},\bm{g},\bm{z}^k, \epsilon^k,\bm{U}^k,\bm{\Theta}^k;\bm{\beta}^k , \bm{\Gamma}^k).
\end{cases}
\end{equation}
Then, using   $(\bm e^{k+1}, \bm g^{k+1})$, the next point
$(\bm z^{k+1}, \epsilon^{k+1}, \bm U^{k+1},\bm\Theta^{k+1})$ is achieved via the following scheme 
\begin{equation}\label{sGS}
		\begin{cases}
		&\widetilde{\bm{z}}^k  := \argmin\limits_{\bm{z}}	\mathcal{L}_{\rho}(\bm{e}^{k+1},\bm{g}^{k+1},\bm{z}, \epsilon^k ,\bm{U}^k,\bm{\Theta}^k;\bm{\beta}^k , \bm{\Gamma}^k),		\\
		& {\tilde\epsilon}^k := \argmin\limits_\epsilon	\mathcal{L}_{\rho}({\bm{e}}^{k+1},\bm{g}^{k+1},\widetilde{\bm{z}}^k, \epsilon ,\bm{U}^k,\bm{\Theta}^k;\bm{\beta}^k , \bm{\Gamma}^k),		\\
		&\widetilde{\bm{U}}^k  := \argmin\limits_{\bm{U}}	\mathcal{L}_{\rho}({\bm{e}}^{k+1},\bm{g}^{k+1},\widetilde{\bm{z}}^k,\tilde\epsilon^k,\bm{U},\bm{\Theta}^k;\bm{\beta}^k , \bm{\Gamma}^k),		\\
		&{\bm{\Theta}}^{k+1}  := \argmin\limits_{\bm{\Theta}}	\mathcal{L}_{\rho}({\bm{e}}^{k+1},\bm{g}^{k+1},\widetilde{\bm{z}}^k ,\tilde\epsilon^k,\widetilde{\bm{U}}^k,\bm{\Theta};\bm{\beta}^k , \bm{\Gamma}^k),\\
		&\bm{U}^{k+1}  := \argmin\limits_{\bm{U}}	\mathcal{L}_{\rho}({\bm{e}}^{k+1},\bm{g}^{k+1},\widetilde{\bm{z}}^k,{\tilde\epsilon}^k,\bm{U},\bm{\Theta}^{k+1};\bm{\beta}^k , \bm{\Gamma}^k),		\\
		&\epsilon^{k+1}  := \argmin\limits_\epsilon	\mathcal{L}_{\rho}({\bm{e}}^{k+1},\bm{g}^{k+1},\widetilde{\bm{z}}^k, \epsilon ,\bm{U}^{k+1} ,\bm{\Theta}^{k+1};\bm{\beta}^k , \bm{\Gamma}^k),		\\
		&\bm{z}^{k+1}  := \argmin\limits_{\bm{z}}	\mathcal{L}_{\rho}({\bm{e}}^{k+1},\bm{g}^{k+1},\bm{z}, \epsilon^{k+1} ,\bm{U}^{k+1} ,\bm{\Theta}^{k+1};\bm{\beta}^k , \bm{\Gamma}^k).
		\end{cases}
\end{equation}
While the latest iterates $(\bm e^{k+1}, \bm g^{k+1},\bm z^{k+1}, \epsilon^{k+1}, \bm U^{k+1},\bm\Theta^{k+1})$ are obtained, the Lagrangian multipliers $(\bm\beta,\bm\Gamma)$ are updated according to the following schemes:
\begin{equation}\label{upmul}
\begin{cases}
	&\bm{\beta}^{k+1} := \bm{\beta}^{k} +\varrho \rho (\bm{g}^{k+1} - \bm{r} + \bm{e}^{k+1} + \tilde{S}\bm{z}^{k+1}),
	\\
	&\bm{\Gamma}^{k+1} := \bm{\Gamma}^{k} + \varrho \rho  \left(  \bm{\Theta}^{k+1} -\left( \begin{array}{cc}
		\mathcal T(\bm{U}^{k+1}) & \bm{z}^{k+1} \\
		(\bm{z}^{k+1})^H& \epsilon^{k+1} 
	\end{array}     \right)   \right),
\end{cases}
\end{equation}
where $\varrho\in(0,(1+\sqrt{5})/2)$ is a step length. 

Comparing the scheme \eqref{geg}-\eqref{sGS} to ADMM of Zheng et al. \cite{zheng2017super}, we see that the visible difference between them is that there are some additional subproblems.
One may naturally think that this approach must lead to more computational burdens, and thus make it  more inefficient.  Next, we will show that this iterative scheme converges globally and that the total iterations would be obviously reduced, so that, these extra computational costs are almost negligible.

\subsection{Subproblems' solving}

This part is devoted to solving the subproblems in \eqref{geg}-\eqref{sGS}. For convenience, we partition the variables $\bm\Theta\in\mathbb{C}^{(MN+1)\times(MN+1)}$ and $\bm\Gamma\in\mathbb{C}^{(MN+1)\times(MN+1)}$ into the following forms
\begin{eqnarray}\label{eq:Gmma}
	\bm{\Theta} \triangleq 
	\begin{pmatrix}
		\bm{\Theta}_0      &   \bm{\theta}_1  \\
		(\bm{\theta}_1)^H    &   \bar{\Theta}
	\end{pmatrix},
	\qquad \bm{\Gamma} \triangleq 
	\begin{pmatrix}
		\bm{\Gamma}_0      &   \bm\gamma  \\
		(\bm\gamma)^H    &   \bar{\Gamma}
	\end{pmatrix},
\end{eqnarray}
where $\bm{\Theta}_0\in\mathbb{C}^{MN \times MN}$ and $ \bm{\Gamma}_0\in\mathbb{C}^{ MN \times MN}$ are submatrices, $\bm{\theta}_1\in\mathbb{C}^{MN} $ and $\bm\gamma\in\mathbb{C}^{MN}$ are column vectors, $\bar{\Theta}$ and $\bar\Gamma $ are scalars.

At the first place, we focus on solving subproblems regarding to the variables $\bm g$, $\bm z$, $\epsilon$, and $\bm U$ because $\L_\rho(\cdot)$ is smooth if $\bm e$ and $\bm \Theta$ are fixed.  For simplicity, we abbreviate the augmented Lagrangian function $\mathcal{L}_{\rho}({\bm{e}},\bm{g},{\bm{z}}, \epsilon,\bm U,\bm{\Theta} ;\bm{\beta} , \bm{\Gamma})$ as $\L_\rho$, and omit the superscripts appeared in \eqref{geg} and \eqref{sGS}. 
It is trivial to deduce that the partial derivative of $\L_\rho(\cdot)$ with respect to variables $\bm g$, $\bm z$, $\epsilon$, and $\bm U$ takes the following explicit form
\begin{subequations}
\begin{align}
		\bigtriangledown_{\bm{g}}\mathcal {L}_{\rho}  &=   (1+\rho)\bm{g} + \bm\beta- \rho(\bm{r}-\bm{e}-\tilde{\bm{S}}\bm{z}),\label{gradg}
		\\
		\bigtriangledown_{\bm{z}}\mathcal{L}_{\rho}  &= \rho  (\tilde{\bm{S}})^H(\tilde{\bm{S}}\bm{z}+ \bm{g} -\bm{r}+\bm{e}+ \frac{\bm{\beta}}{\rho})  + 2\rho(\bm{z}-\bm{\theta}_1) - 2\bm\gamma,\label{gradz}
		\\
		\bigtriangledown_ \epsilon\mathcal{L}_{\rho}  &=  \frac{\lambda}{2} + \rho ( \epsilon-\bar{\Theta}) - \bar{\Gamma},\label{grade}
		\\
		\bigtriangledown_{u_{l(m)}}\mathcal{L}_{\rho}  &= 
		\begin{cases}
			\frac{\lambda}{2} + MN\rho u_0(0) - {\rm Tr}(\rho \bm{\Theta}_0 + \bm{\Gamma}_0  ), & \text{if} \ l=m=0,\\[2mm]
			(N-|l|)(M- |m|)\rho u_l(m) - \sum^{N-|l|}_{j=1}{\rm Tr}_m(\mathcal{S}_{l,j}(\rho \bm{\Theta}_0 + \bm{\Gamma}_0  )),                                                      &\text{if} \ l\neq 0 ~\text{or}~ m \neq 0.\label{gradu}
		\end{cases}
\end{align}
\end{subequations}
For any $\bm{P}\in \mathbb{C}^{MN \times MN}$, we use the operator $\mathcal{S}_{l,j}(\bm{P})$ to return the $(l,j)$--th $M\times M$ submatrix $\bm{P}_{l,j}$, where the position of matrix $\bm{P}_{l,j}$ is corresponding to the $j$--th block matrix \text{Toep}($\bm u_l$) in $\mathcal{T}(\bm U)$ (from left to right), and then use the operator ${\rm Tr}_m(\cdot)$ to output the sum of the elements locating in $m$-th sub-diagonal of an input matrix. We note that  \eqref{gradg}-\eqref{grade} are actually linear systems, which indicates that the iterates $(\bm g^{k+1},\bm z^{k+1}, \epsilon^{k+1})$ can be got easily. 
We now pay our attention to the linear system \eqref{gradu} to get the latest $\bm U^{k+1}$ in which its value is related to the position of the elements in $\bm U$. 
Let $\bm{I}_{(M,N)}$ be a $(2M-1)\times (2N-1)$ complex matrix  whose element is $1$ at $(M, N)$-th position and $0$ otherwise. Using this matrix, the solution to $\bigtriangledown_{\bm U}\mathcal{L}_{\rho} = 0$ can be expressed as the following unified form: 
\begin{equation*}
	\bm{U} = \mathcal{T}^*\Big(\bm{\Theta}_0 + \frac{\bm{\Gamma}_0}{\rho} \Big) - \frac{\lambda}{2MN\rho}\bm{I}_{(M,N)},
\end{equation*}
where $\mathcal{T}^*(P) = \bm Q :=  (\bm q_{-N+1},\bm q_{-N+2},\ldots,\bm q_{N-1})
$ with 
$$
\bm q_l=(q_l(-M+1),q_l(-M+2),\ldots,q_l(M-1))^{\top}
$$ 
for any $l= -N+1,-N+2,\dots,N-1$, and 
\begin{equation*}
	q_l(m) := \frac{1}{(N-|l|)(M- |m|)}  \sum^{N-|l|}_{j=1}{\rm Tr}_m(\mathcal{S}_{l,j}(P))
\end{equation*}
for any $m= -M+1, -M+2,\dots,M-1$. 

At the second place, we focus on solving the $\bm g$-subproblem and $\bm\Theta$-subproblem involved in \eqref{geg} and \eqref{sGS}, respectively.
Using Definition \ref{def2.2}, we see that the $\bm e$- and $\bm\Theta$-subproblems obey the forms of proximal point mapping, and thus,  admit analytic solutions, that is,
\begin{align*}
	\bm{e} =~&  {\text{Prox}}^{\mu/\rho}_{\|\cdot\|_1}(\bm{r} -\bm{g}   - \tilde{S}\bm{z}-{\bm\beta}/{\rho})\\
	=~&\text{sign}(\bm{r} -\bm{g}   - \tilde{S}\bm{z}-{\bm\beta}/{\rho})\odot \max\left\{|\bm{r} -\bm{g}   - \tilde{S}\bm{z}-{\bm\beta}/{\rho}|-{\mu}/{\rho} , 0\right\},
\end{align*}
and
$$
\bm\Theta =\Pi_{\mathcal{S}^{MN+1}_+}\left(
\begin{pmatrix}
	\mathcal T(\bm U) & \bm z \\
	\bm{z}^H  & \epsilon
\end{pmatrix}
- \frac{\bm\Gamma}{\rho} \right).
$$
Let $P\in\mathbb{C}^{(MN+1)\times (MN+1)}$ be a matrix with its eigenvalue decomposition such that
\begin{eqnarray*}
	P\triangleq
	\begin{pmatrix}
		\mathcal T(\bm U) & \bm z \\
		\bm{z}^H  & \epsilon
	\end{pmatrix}
	- \frac{\bm\Gamma}{\rho},\quad\text{and}\quad
		P\triangleq V\begin{pmatrix}
		\zeta_1 &~ &~ &~ \\
		~& \zeta_2&~&~\\
		~&~&\ddots&~\\
		~&~&~&  \zeta_{MN+1}
	\end{pmatrix}V^H,
\end{eqnarray*}
where $\zeta_i$ is an eigenvalue and $V$ is a matrix with its $i$-th column being the eigenvector correspondingly.
From optimization literature, it is easy to see that the projection of $P$ over $\S^{MN+1}_+$ admits a compact form, that is,
\begin{eqnarray*}
\Pi_{\mathcal{S}^{MN+1}_+}(P)= V\begin{pmatrix}
	\max\{0,\zeta_1\} &~ &~ &~ \\
	~& \max\{0,\zeta_2\}&~&~\\
	~&~&\ddots&~\\
	~&~&~&  \max\{0,\zeta_{MN+1}\}
\end{pmatrix}V^H.
\end{eqnarray*}

In light of the above analysis, we are ready to state the full steps of sGS based ADMM when it is employed to solve the problem  \eqref{eq:mod3}.

\begin{framed}	
\noindent{\bf Algorithm $1$: sGS-ADMM}
\vskip 1.0mm \hrule \vskip 1mm	
	\begin{itemize}			
	\item[Step 0.] Let $\rho > 0$ and $\varrho  \in (0,\frac{1+\sqrt{5}}{2})$ be given parameters. Choose an initial point $(\bm{e}^0,\bm{g}^0,\bm{z}^0, \epsilon^0, \bm {U}^0,\bm{\Theta}^0 ;\bm{\beta}^0,\bm{\Gamma}^0) $. For $k=0,1,2,\dots$,  generate the sequence $\{(\bm e^{k},\bm{g}^{k},\bm{z}^{k}, \epsilon^{k}, \bm U^{k},\bm\Theta^{k};\bm{\beta}^{k},\bm{\Gamma}^{k} )\}$ according to the following iterations until a certain stopping criterion is met.
		\item [Step 1.] Compute $(\bm e^{k+1},\bm g^{k+1})$ via the following steps:
		\begin{eqnarray*}
			\widetilde{\bm{g}}^k  &:= & \frac{\rho}{1+\rho}   (\bm{r}-\bm{e}^k-\widetilde{\bm{S}}\bm{z}^k)-\frac{1}{1+\rho}\bm{\beta}^k,
			\\
			\bm{e}^{k+1} &:= & {\text{Prox}}^{\mu/\rho}_{\|\cdot\|_1}(\bm{r} - \widetilde{\bm{g}}^k   - \tilde{S}\bm{z}^k-\frac{\bm{\beta}^k}{\rho}),
			\\
			\bm{g}^{k+1} &:= & \frac{\rho}{1+\rho}
			(\bm{r}-\bm{e}^{k+1}-\tilde{\bm{S}}\bm{z}^k)-\frac{1}{1+\rho}\bm{\beta}^k.
		\end{eqnarray*}
		\item [Step 2.] Compute $(\bm z^{k+1},\epsilon^{k+1},\bm U^{k+1},\bm\Theta^{k+1})$ via the following two steps:
		\begin{itemize}
			\item[Step 2.1] Compute the temporary points $(\bm\tilde z^{k},\tilde\epsilon^{k},\bm\tilde U^{k})$ via the following three steps:
		\begin{eqnarray*}
			\widetilde{\bm{z}}^k &:= & (\tilde{\bm S}^H\tilde{\bm S}+2\bm{I}_{MN})^{-1}[\tilde{\bm S}^H(\bm{r} - \bm{g}^{k+1}  - \bm{e}^{k+1}-\frac{\bm{\beta}^k}{\rho}) + 2\bm{\theta}^k_1 + \frac{2}{\rho}\bm\gamma^k],
			\\
			{\tilde{ \epsilon}}^k  &:= & \frac{1}{\rho} {\bar\Gamma}^k + {\bar \Theta}^k - \frac{\lambda}{2\rho},
			\\
			\widetilde{\bm{U}}^k &:= & \mathcal{T}^*(\bm{\Theta}^k_0 + \frac{\bm{\Gamma}^k_0}{\rho} )- \frac{\lambda}{2MN\rho}\bm{I}_{(M,N)}.	
		\end{eqnarray*}
		\item [Step 2.2.]	Compute $(\bm z^{k+1},\epsilon^{k+1},\bm U^{k+1},\bm\Theta^{k+1})$ via the following four steps:
		\begin{eqnarray*}
			\bm{\Theta}^{k+1} &:= & \Pi_{\mathcal{S}^{(MN+1)}_+}\left(  
			\begin{pmatrix}
				\mathcal T(\widetilde{\bm{U}}^k) & \widetilde{\bm{z}}^k \\
				(\widetilde{\bm{z}}^k)^H  & {\tilde\epsilon}^k 
			\end{pmatrix}
			- \frac{\bm{\Gamma}^k}{\rho} \right) ,\\
			\bm{U}^{k+1} & :=  & \mathcal{T}^*(\bm{\Theta}^{k+1}_0 + \frac{\bm{\Gamma}^k_0}{\rho}) - \frac{\lambda}{2MN\rho}\bm{I}_{(M,N)},	\\
			\epsilon^{k+1} &:= & \frac{1}{\rho} \bar\Gamma^k + \bar\Theta^{k+1} - \frac{\lambda}{2\rho},
			\\
			\bm{z}^{k+1}  & := & (\tilde{\bm S}^H\tilde{\bm S}+2\bm{I}_{MN})^{-1}[\tilde{\bm S}^H(\bm{r} - \bm{g}^{k+1}  - \bm{e}^{k+1}-\frac{\bm{\beta}^k}{\rho}) + 2\bm{\theta}^{k+1}_1 + \frac{2}{\rho}\bm\gamma^{k+1}].
		\end{eqnarray*}
			\end{itemize}
		\item [Step 3.] Update the multipliers $\bm \beta$ and $\bm\Gamma$ via
		\begin{eqnarray*}
			\bm{\beta}^{k+1} &:=  &\bm{\beta}^{k} +\varrho \rho (\bm{g}^{k+1} - \bm{r} + \bm{e}^{k+1} + \tilde{S}\bm{z}^{k+1}),
			\\
			\bm{\Gamma}^{k+1}&:=  &\bm{\Gamma}^{k} +\varrho \rho \left(  \bm{\Theta}^{k+1} -
			\begin{pmatrix}
				\mathcal T(\bm{U}^{k+1})  &  \bm{z}^{k+1} \\
				(\bm{z}^{k+1})^H    &   \epsilon^{k+1} 
			\end{pmatrix}   \right).
		\end{eqnarray*}
	\end{itemize}
\vspace{-.5cm}
\end{framed}

We see that Algorithm $1$ is easily implementable because the analytic solutions are permitted for each variable.
But there is an inverse operation involved in  $\tilde{\bm{z}}^{k}$ and $\bm{z}^{k+1}$ which may take significant time in the case of $MN$ are especially large. In our numerical experiments part, we use an iterative method to find an approximate solution so that the matrix inverse is avoided. At last, we must point out that the iterative process should be terminated when the Karush-Kuhn-Tucker (KKT) system associated with the current point is small enough. For more detail on this termination criteria, we can refer to the numerical experiments' part in Subsection \ref{kktpart}. 

\subsection{Convergence results}
In this part, we establish the convergence result
of Algorithm $1$ by means of the relationship with the semi-proximal ADMM  of Fazel et al. \cite{fazel2013hankel}. The follow lemma shows that the sGS used individually in the groups of $(\bm e,\bm g)$ and $(\bm z,\epsilon,\bm U,\bm\Theta)$ is equivalent to solving these variables together but with an additional semi-proximal term.

\begin{lemma}\label{lem31}
	For any $k \geq 0$, the $(\bm{e}, \bm g)$-subproblem in Step $1$ and $(\bm z, \epsilon, \bm U, \bm\Theta)$-subproblem in Step $2$ in Algorithm $1$ can be respectively expressed as the following compact form:
		\begin{align}
		({\bm{e}}^{k+1} ,\bm{g}^{k+1})=& \argmin\limits_{{\bm{e}},\bm{g}}	\mathcal{L}_{\rho}({\bm{e}},\bm{g},{\bm{z}}^k, \epsilon^k ,\bm{U}^k,\bm{\Theta}^k;\bm{\beta}^k , \bm{\Gamma}^k)+\frac12\|\bm Y-\bm Y^k \|_{\T_1}^2,\label{Y}\\
		(\bm{z}^{k+1} , \epsilon^{k+1},\bm{U}^{k+1},\bm{\Theta}^{k+1})  =&\argmin\limits_{{\bm{z}}, , \epsilon,\bm{U},\bm{\Theta}}	\mathcal{L}_{\rho}({\bm{e}}^{k+1},\bm{g}^{k+1},{\bm{z}}, \epsilon ,\bm{U},\bm{\Theta};\bm{\beta}^k , \bm{\Gamma}^k)+\frac12\|\bm W- \bm W^k \|_{\T_2}^2,\label{W}
	\end{align}
where $\T_1$ and $\T_2$ are self-adjoint semi-positive definite linear operators.
\end{lemma}
\begin{proof}
	It is sufficient to note that the Hessian, denoted by $\bm Q_Y$,
   of  $\mathcal{L}_{\rho}(\bm{e},\bm{g},\bm{z}^k,\epsilon^k,\bm{U}^k,\bm{\Theta}^k;\bm{\beta}^k , \bm{\Gamma}^k)$ regarding to variables $(\bm e, \bm g)$ is 
	given by
\begin{eqnarray*}
	\bm Q_{\bm Y}:=\rho \begin{pmatrix}
		\bm I_{MN+1}   &  \bm I_{MN+1}\\
		\bm I_{MN+1}   &   \frac{\rho+1}{\rho} \bm I_{MN+1}
	\end{pmatrix} = \bm M_{\bm Y} + \bm D_{\bm Y} +\bm M^H_{\bm Y},
\end{eqnarray*}
where $\bm M_Y$ and $\bm D_Y$ be the strictly upper triangular and diagonal of $\bm Q_Y$, respectively, that is,
\begin{eqnarray*}
	\bm M_{\bm Y} := \begin{pmatrix}
		\bm 0_{MN+1}   &  \rho\bm I_{MN+1}\\
		\bm 0_{MN+1}   &    \bm 0_{MN+1}
	\end{pmatrix},  \qquad 
	\bm D_{\bm Y}:= \begin{pmatrix}
		\rho \bm I_{MN+1}   &  \bm 0_{MN+1}\\
		\bm 0_{MN+1}   &   (\rho+1) \bm I_{MN+1}
	\end{pmatrix}.
\end{eqnarray*}
Then, from the sGS decomposition theorem of Li et al. \cite{li2019block}, it is easy to see that the $\T_1$ obeys the following form:
\begin{eqnarray}\label{tt1}
	\bm {\T_1}:= \bm M_{\bm Y} \bm D^{-1}_{\bm Y}\bm M^H_{\bm Y}= \begin{pmatrix}
		\frac{\rho^2}{1+\rho} \bm I_{MN+1}   &  \bm 0_{MN+1}\\
		\bm 0_{MN+1}   &   \bm 0_{MN+1}
	\end{pmatrix}.
\end{eqnarray}
In a similar way, let $\bm Q_{\bm W}$ be the Hessian of  $\mathcal{L}_{\rho}(\bm{e}^j,\bm{g}^k,\bm{z}, \epsilon,\bm{U},\bm{\Theta};\bm{\beta}^k , \bm{\Gamma}^k)$ regarding to variables $(\bm{z}, \epsilon,\bm{U},\bm{\Theta})$, and then partition it into
\begin{eqnarray*}
	\bm Q_{\bm W}:=\rho\begin{pmatrix}
		\bm I_{MN+1}   &  -\bm I_{MN+1}    & -\bm I_{MN+1}\\
		- \bm I_{MN+1}   &  \bm B    & \bm I_{MN+1}\\
		- \bm I_{MN+1}   &   \bm I_{MN+1}    & \bm I_{MN+1}
	\end{pmatrix} = \bm M_{\bm W} + \bm D_{\bm W} +\bm M^H_{\bm W},
\end{eqnarray*}	
where 
\begin{eqnarray*}
	    \bm M_{\bm W}:=\rho\begin{pmatrix}
		\bm 0_{MN+1}   &  -\bm I_{MN+1}    & -\bm I_{MN+1}\\
		\bm 0_{MN+1}   &   \bm 0_{MN+1}    & \bm I_{MN+1}\\
		\bm 0_{MN+1}   &   \bm 0_{MN+1}    & \bm 0_{MN+1}
	\end{pmatrix} , \quad 
 	\bm D_{\bm W}:=\rho\begin{pmatrix}
 	\bm E_{MN+1}   &  \bm 0_{MN+1}    & \bm 0_{MN+1}\\
 	\bm 0_{MN+1}   &  \bm B    & \bm 0_{MN+1}\\
 	\bm 0_{MN+1}   &   \bm 0_{MN+1}    & \bm I_{MN+1}
 \end{pmatrix},
\end{eqnarray*}
with
\begin{eqnarray*}
	 \bm B := \begin{pmatrix}
	 	\tilde{\bm S}^H\tilde{\bm S} +\bm I_{MN} & \bm 0\\
	 	\bm 0^\top & 1
	 \end{pmatrix}.
\end{eqnarray*}
Also,  from the sGS decomposition theorem of Li et al. \cite{li2019block}, it yields that
\begin{eqnarray}\label{tt2}
	\bm{\T_2} := \rho\begin{pmatrix}
		\bm B^{-1} + \bm I_{MN+1}  &  -\bm I_{MN+1}    & \bm 0_{MN+1}\\
		- \bm I_{MN+1}           &   \bm I_{MN+1}    & \bm 0_{MN+1}\\
		\bm 0_{MN+1}             &   \bm 0_{MN+1}    & \bm 0_{MN+1}
	\end{pmatrix} .
\end{eqnarray}	
Therefore, the desired result of this lemma is proved.
\end{proof}

From this lemma, we know that the iteration scheme  \eqref{geg} is equivalent to \eqref{Y} if $\T_1$ is given in \eqref{tt1}, and \eqref{sGS} is equivalent to \eqref{W} if $\T_2$ is given in \eqref{tt2}. This equivalence is essential because it reduces the minimizing task of $\mathcal{L}_{\rho}({\bm{e}},\bm{g},{\bm{z}}, \epsilon,\bm U,\bm{\Theta} ;\bm{\beta}^k , \bm{\Gamma}^k)$ to two parts of $(\bm e,\bm g)$ and $(\bm z,\epsilon,\bm U,\bm\Theta)$. 
As a result, the convergence of Algorithm $1$ can be easily followed by the known
convergence result in optimization. To conclude this part, we state the convergence of Algorithm $1$ listed below. For its proof, one may refer to  \cite[TheoremB.1]{fazel2013hankel}.
\begin{theorem}\label{the31}
	Let the sequence $\{(\bm e^{k},\bm{g}^{k},\bm{z}^{k}, \epsilon^{k}, \bm U^{k},\bm\Theta^{k};\bm{\beta}^{k},\bm{\Gamma}^{k} )\}$ be generated by Algorithm $1$ for any initial point $(\bm e^{0},\bm{g}^{0},\bm{z}^{0}, \epsilon^{0}, \bm U^{0},\bm\Theta^{0};\bm{\beta}^{0},\bm{\Gamma}^{0} )$  with
	$\varrho \in(0, \frac{1+\sqrt{5}}{2})$, then it converges to the accumulation
	point $(\bar{\bm e},\bar{\bm{g}},\bar{\bm{z}}, \bar \epsilon, \bar{\bm U},\bar{\bm\Theta}; \bar{\bm\beta}, \bar{\bm{\Gamma}} )$ such that $(\bar{\bm e},\bar{\bm{g}},\bar{\bm{z}}, \bar \epsilon, \bar{\bm U},\bar{\bm\Theta} )$ is the solution of the problem \eqref{eq:mod3}, and that $(\bar{\bm\beta}, \bar{\bm{\Gamma}} )$ is an optimal solution of the dual problem  \eqref{eq:dual}.
\end{theorem}
	
At the end of this section, we are devoted to returning  the delay $\bm \psi$ and the  Doppler frequency $\bm \phi$ when an optimal solution $(\bar{\bm e},\bar{\bm{g}},\bar{\bm{z}}, \bar \epsilon, \bar{\bm U},\bar{\bm\Theta} )$ of the problem \eqref{eq:mod3} is achieved.  
From \eqref{eq:R}, we know that $\bm \psi$ and  $\bm \phi$ are closely related to  $\bm z$. However, as shown in (\ref{eq:z}) that, the relationship between them seems more  complicated so that it is not a trivial task to derive them only from $\bm z$.
Alternatively, we adopt the aforementioned MUSIC which is very popular in frequency estimation because of its better resolution's quality and higher accuracy  \cite{naha2014determining}.
More precisely, when  $\bar{\bm z}$ is obtained by Algorithm $1$, we utilize the orthogonality of signal subspace and noise subspace to construct a spatial spectral function, and then, we estimate the signal parameters by locating the poles in the spectrum.  
It should be noted that the  parameters estimated by MUSIC include the information of all paths. If we want to retain the  parameters related to targets, we need to perform a filtering operation according to the delay of different paths. Typically, the delay of clutter is small and the delay of direct path is zero. More details are omitted here, because they are beyond the scopes of this paper.

\section{Numerical Experiments}\label{sec4}

In this section, we conduct some numerical experiments to illustrate the accuracy and efficiency of Algorithm $1$ (named sGS-ADMM) while it is used to estimate  moving targets according to solving the SDP problem \eqref{eq:mod2}.  In particular, we mainly focus on evaluating the numerical performance of Algorithm $1$, and particularly do comparisons to the directly-extended ADMM  of Zheng et al. \cite{zheng2017super}.
Because we take less notice of the tuning parameters $\lambda$ and $\mu$ in the model \eqref{eq:mod2} while are used for different data, here, we fix their values based on a scalar $\sigma$, that is,
\begin{displaymath}
	\lambda := \sigma \sqrt{MN \log(MN)},	~~\text{and}~~\mu := \frac{\lambda}{\sqrt{MN}},
\end{displaymath}
where $\sigma>0$ is a scalar given in advance. 

\subsection{Visual experiments using a simple simulated data}

In the first test, we consider the scenario with multiple targets and clutters, and particularly set the number of targets to $5$ and the number of clutters to $100$. Their velocity and range of these targets are shown in Fig. \ref{fig:target}, where the red ``$\diamond$" represents the targets,  the black ``$\times$" represents the clutters, and the blue ``$\triangle$" represents the direct path.  
In this test, the targets and clutters are assumed to be point scatters which generated randomly in the range of $[0, 30$]km. Besides, the velocities of targets are  within $[-148,148]m/s$ and the velocities of clutters are limited into $[-3, 3]m/s$. 
We assume that there is only one illuminator to transmit OFDM signal with a carrier frequency of $2$ GHZ. 
The length of OFDM symbol is $T=200 \mu s$ and the length of the cyclic prefix is $T_{cp}=100 \mu s$.
We also assume that the number of subcarriers is  $50$, and the signal is divided into $N=16$ blocks for processing. 
In this test, all the algorithms are compiled with Matlab R2016a and run on a LENOVO laptop with one Intel Core Silver 4216 CPU Processor ($2.10$ to $3.20$ GHz) and $32$ GB RAM.
  \begin{figure}
	\centering
	\includegraphics[width=3.5in]{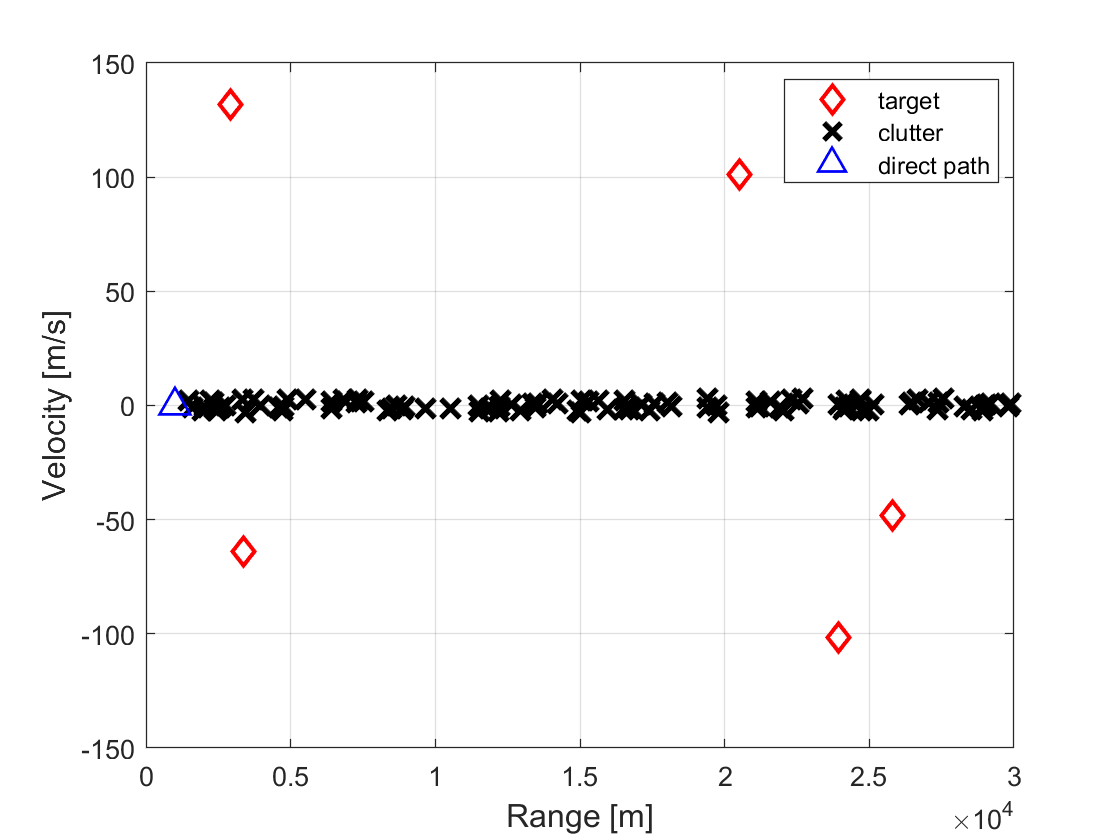}
	\caption{Simulation scenarios with 100 clutters and 5 targets}
	\label{fig:target}
\end{figure}

Taking the demodulation error into account, we employ the quadrature phase-shift keying (QPSK)  to modulate, and use the bit-error-rates (BER) to measure demodulation error. 
In this test, we use sGS-ADMM and directly-extended ADMM, and then compare their performance for different BER from $0$ to $0.1$ with $0.02$ intervals. 
For each algorithm, the iterative process forcefully terminates when the number of iterations reaches $100$, $200$, and $300$, to observe the algorithms' performance within a fixed number of iterations. 
To get a distinct visual comparison, we plot the spectrum for estimation results derived from each algorithm in Figure \ref{fig:ber}. 
It can be clearly seen from this figure that there is no significant difference between sGS-ADMM and ADMM in the spectrum when BER is small. 
But, with the increases of BER, the estimations' qualities produced by ADMM become lower and lower.
At the case of BER$=0.1$, we see from the 2nd-column of Figure \ref{fig:ber} that, many ghost peaks appeared in the spectrum, which may cause many difficulties to identify targets that lead to  increasing the false alarms. 
In contrast, the spectrum displayed at the 1st-column is clearly accurate and ``clean",  which shows that sGS-ADMM is able to detect these targets and does not degrade with the increases of BER. 

 \begin{figure}
	\centering
	\subfloat[][]{\includegraphics[width=3in]{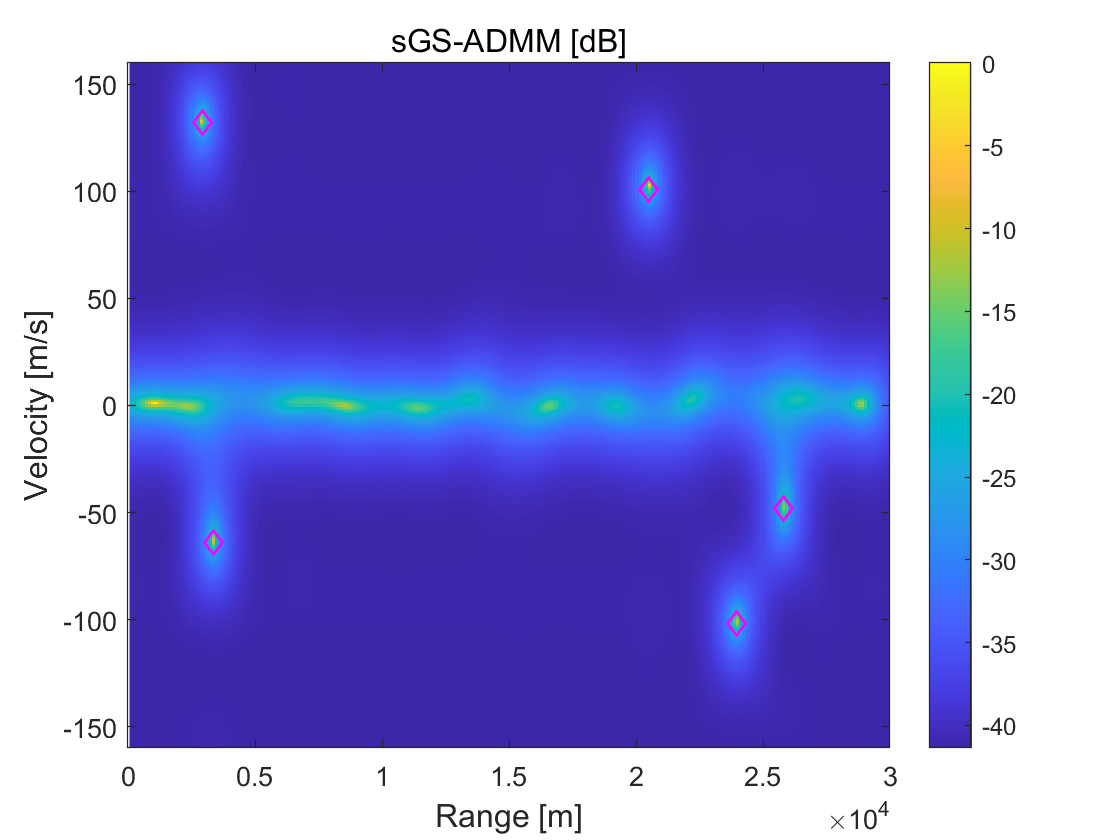}}
	\subfloat[][]{\includegraphics[width=3in]{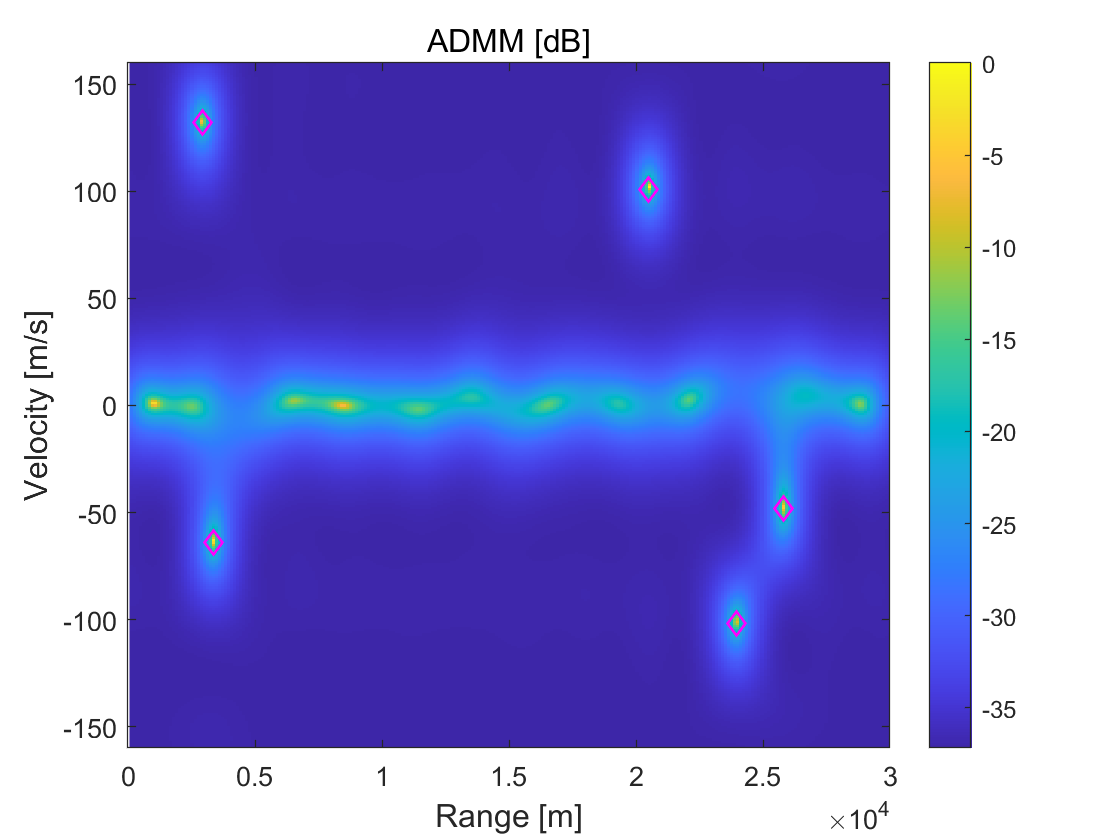}}
	\newline
	\subfloat[][]{\includegraphics[width=3in]{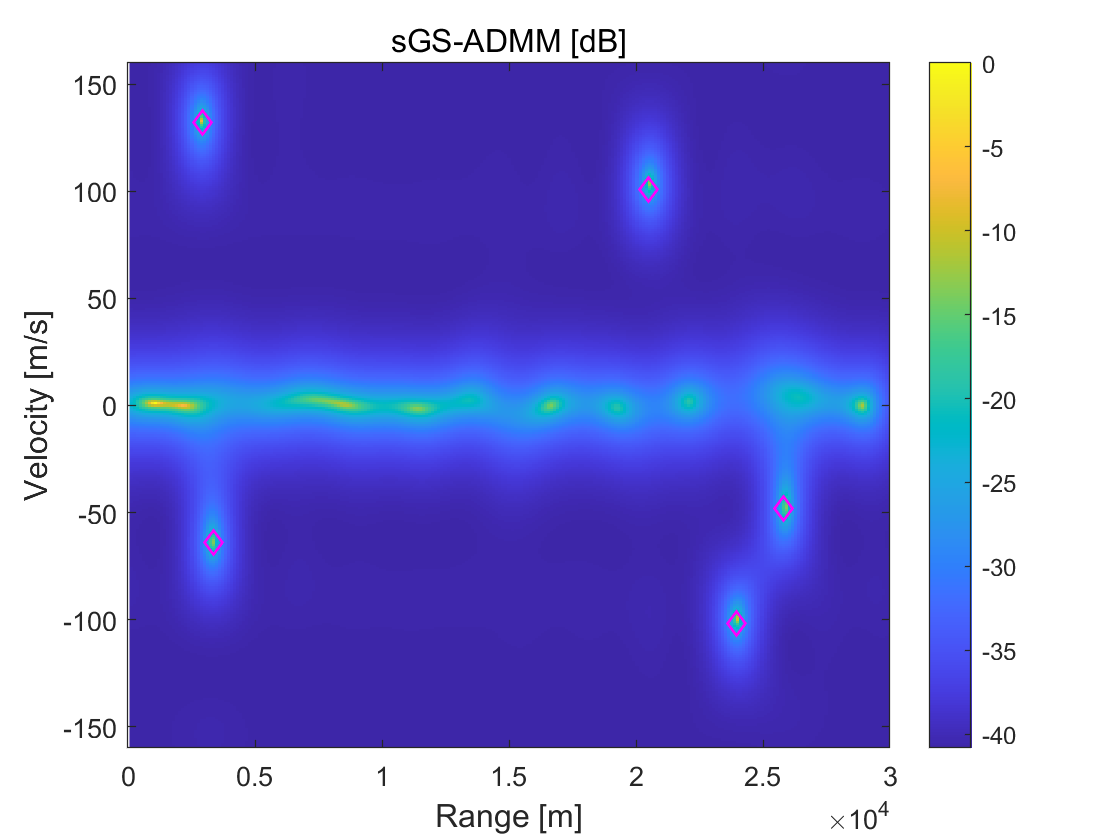}}
	\subfloat[][]{\includegraphics[width=3in]{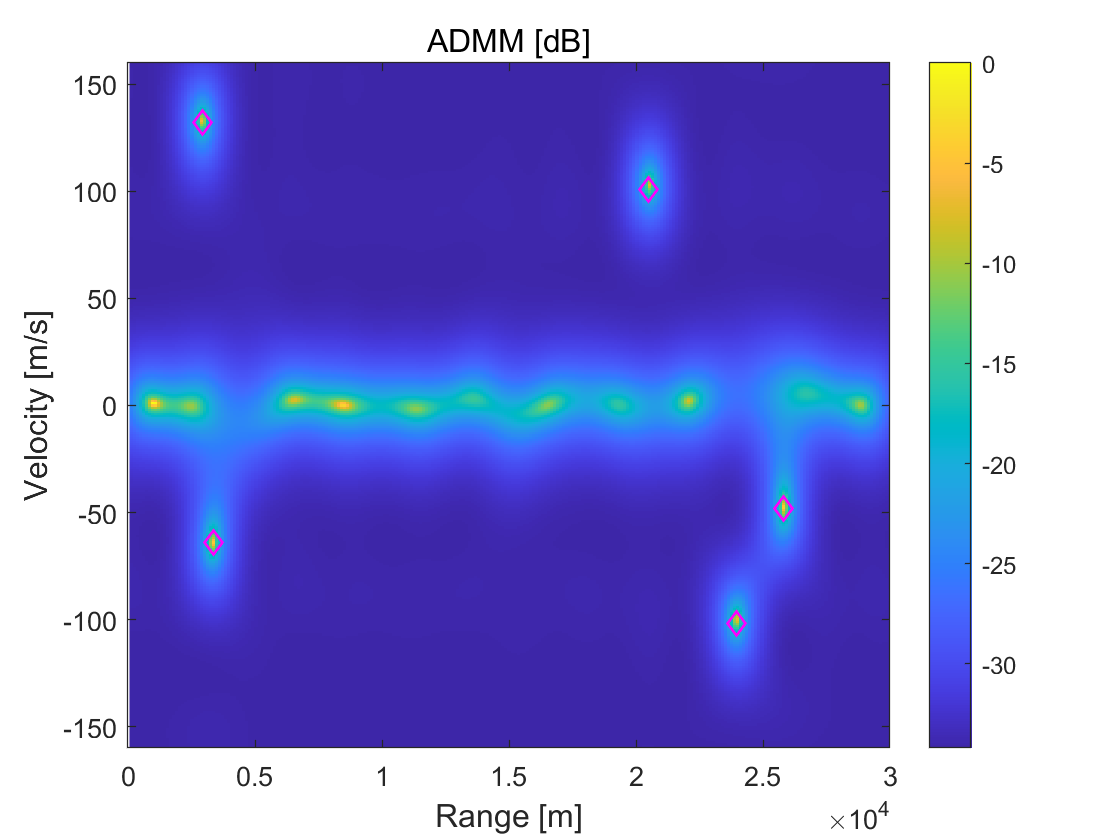}}
	\newline
	\subfloat[][]{\includegraphics[width=3in]{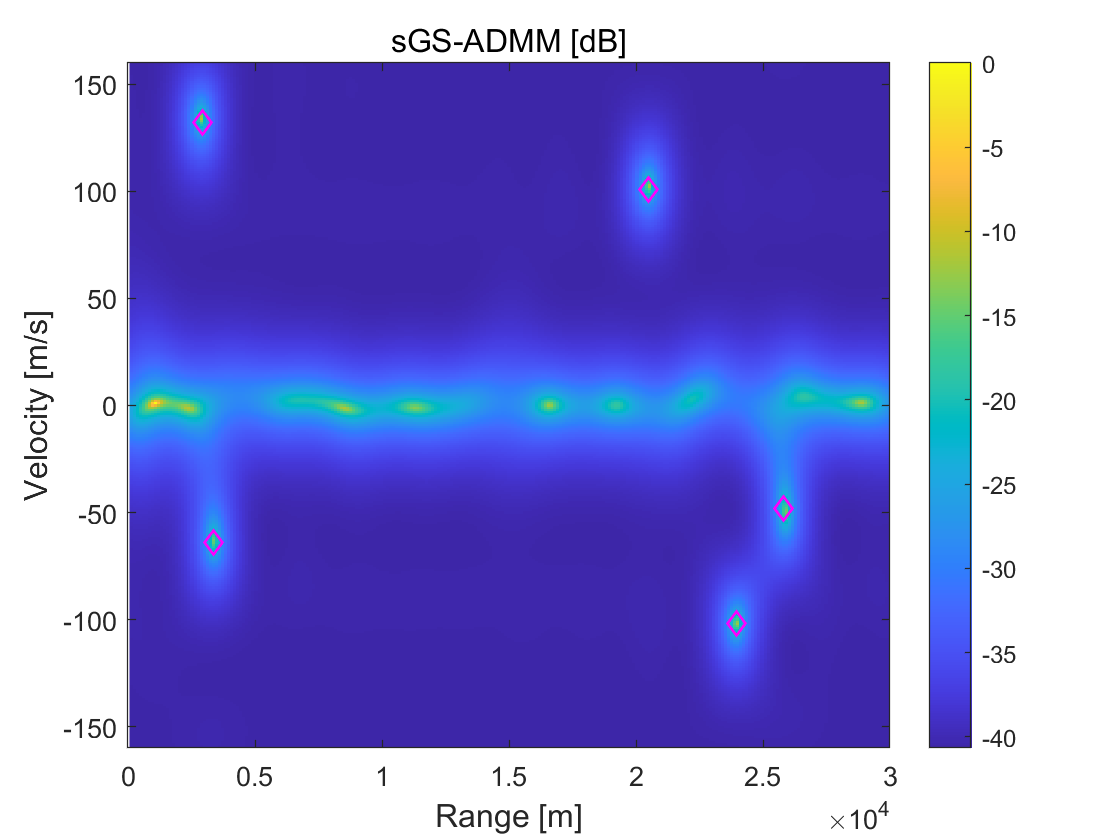}}
	\subfloat[][]{\includegraphics[width=3in]{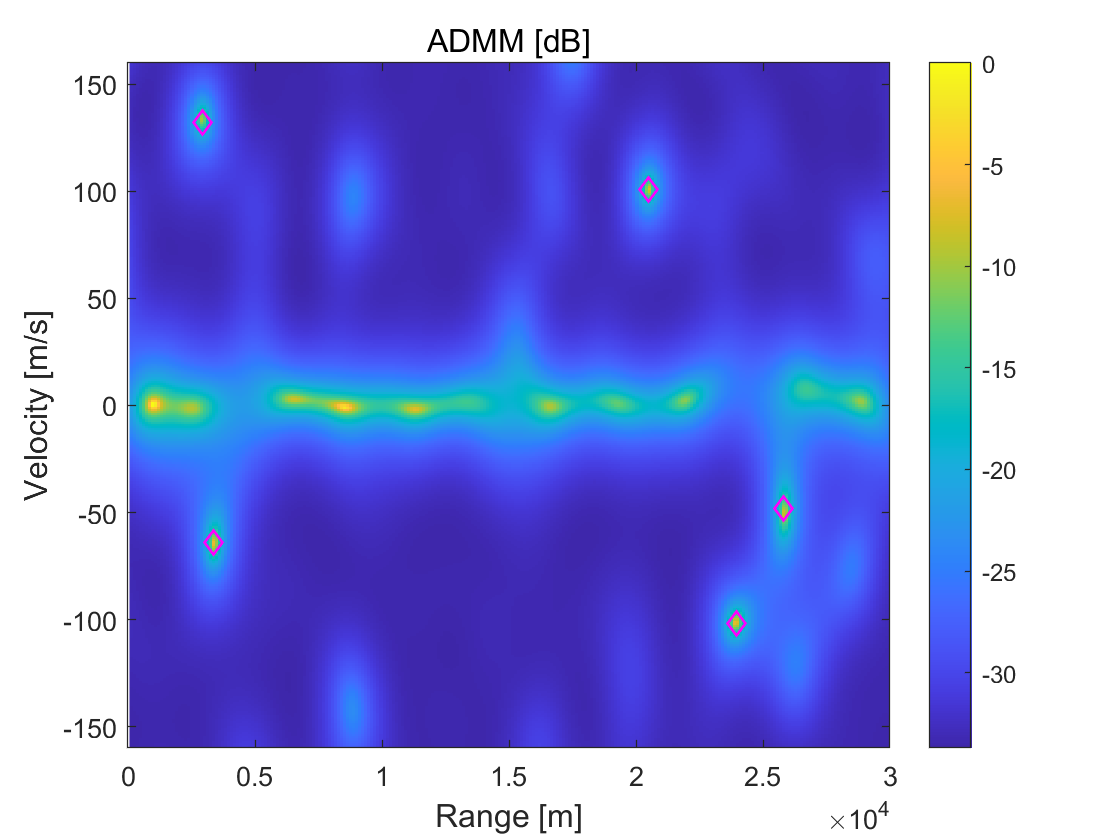}}
	
	\caption{Comparisons of range and velocity accuracy with BER=$0$ (top), BER=$0.06$ (middle), and BER=$0.1$ (bottom).}
	\label{fig:ber}
\end{figure}

\subsection{Numerical performance comparisons}\label{kktpart}
To evaluate the numerical efficiencies of each algorithm,  we adopt the relative KKT residual to measure the estimation qualities associated with the problem \eqref{eq:mod3}. More precisely, the KKT residual, denoted by $\eta_{\max}$, is defined as 
$$
\eta_{\max}:=\max \{\eta_1,\eta_2,\eta_3,\eta_4,\eta_5,\eta_6
 \},
$$
where
$$ 
 \begin{array}{lll}
 	&\eta_1 := \dfrac{\|\bm z-\bm r + \bm e -2\bm\gamma\|_2}{1+\|\bm z\|_2+2\|\bm\gamma\|_2},
 	&\eta_2 := \dfrac{\|\bm e- {\text{Prox}}^{\mu}_{\|\cdot\|_1}( \bm e -\bm{z})\|_2}{1+\|\bm e\|_2 + \|\bm r-\bm e-\bm z\|_2},\\[4mm]
 	&\eta_3 := \dfrac{\|\bm{\Theta}- \Pi_{\mathcal{S}^{(MN+1)}_+}( \bm\Theta -\bm\Gamma)\|_F}{1+\|\bm\Theta\|_F + \|\bm\Gamma\|_F},
 	&\eta_4 :=\dfrac{\|\frac{\lambda}{2}-\bar{\Gamma}\|_2}{1+\|\epsilon\|_2 +\|\bar{\Gamma}\|_2},\\[4mm]
 	&\eta_5 := \dfrac{\|\mathcal{T}^*(\bm{\Gamma_0})-\frac{\lambda}{2MN}\bm I_{(M,N)}\|_F}{1+\|\bm U\|_F+\|\mathcal{T}^*(\bm{\Gamma_0})\|_F},
 	&\eta_6 
 	:= \|\bm{\Theta} -	\begin{pmatrix}
 		\mathcal T(\bm{U}) & \bm{z} \\
 		\bm{z}^H  & \epsilon
 	\end{pmatrix}\|_F.
 \end{array}
$$
From optimization theory, we know that $\eta_{\max}=0$ if only if $(\bar{\bm e},\bar{\bm{g}},\bar{\bm{z}}, \bar \epsilon, \bar{\bm U},\bar{\bm\Theta} )$ is the solution of the problem \eqref{eq:mod3}. Hence, it is reasonable to stop the iterative process of both algorithms when $\eta_{\max}$ is smaller than a given  tolerance.

\begin{table}[h]
	\centering
	\caption{Comparison results of sGS-ADMM and ADMM with fixed steps of $100$, $200$ and $300$}
	\label{tab:maxIter123}
	\setlength{\tabcolsep}{3pt}
		{\small
	\begin{tabular}{ c| c c c  c c  }
		\hline
		\multirow{2}{*} ~step &BER&
		Time(s) &$\eta_{\max}$  &obj & n\_{tar} 
		\\ 
		&&  (sGSA$|$dA) &  (sGSA$|$dA) & (sGSA$|$dA) & (sGSA$|$dA)\\
		\hline
             &0& 53.84 $|$34.35 &  3.29e-04$|$1.29e-03  & 4.77e+03 $|$ 7.22e+04  &5$|$5  \\
	100&  0.02 &  56.32 $|$36.81 &  3.20e-04$|$2.44e-03  & 5.25e+03 $|$ 7.83e+04  &5$|$5  \\
	      &0.04& 54.89 $|$36.46 &  3.49e-04$|$2.02e-03  & 6.00e+03 $|$ 8.32e+04  &5$|$5  \\
          &0.06& 54.66$|$36.45&  3.68e-04$|$2.32e-03  & 5.80e+03 $|$ 8.25e+04  &5$|$5  \\
	      &0.08& 55.07 $|$ 36.67 &  3.89e-04$|$2.53e-03  & 6.98e+03 $|$ 8.99e+04  &6$|$5  \\
	       &0.1& 55.07 $|$ 36.70 &  4.17e-04$|$2.62e-03  & 7.71e+03 $|$ 8.45e+04  &5$|$6  \\
		\hline 	
		     &0& 108.20 $|$69.24 &  7.73e-05$|$3.48e-04  & 3.42e+03 $|$ 7.28e+04  &5$|$5  \\
	200&  0.02 & 112.17 $|$83.68 &  8.07e-05$|$3.94e-04  & 3.90e+03 $|$ 7.90e+04  & 5$|$5  \\
		  &0.04& 109.79 $|$72.38 &  8.45e-05$|$4.03e-04  & 4.60e+03 $|$ 8.46e+04  &5$|$5  \\
		  &0.06& 110.10 $|$73.76 &  9.40e-05$|$4.13e-04  & 4.42e+03 $|$ 8.37e+04  &5$|$5  \\
	      &0.08& 110.25 $|$ 73.45 &  9.00e-05$|$4.19e-04  & 5.60e+03 $|$ 9.07e+04  &5$|$5  \\
	       &0.1& 110.46 $|$74.26 &  9.01e-05$|$4.95e-04  & 6.21e+03 $|$ 8.20e+04  &5$|$10  \\
		\hline 	
		    &0& 162.20 $|$103.47 &  3.68e-05$|$ 1.92e-04  & 2.89e+03 $|$ 7.30e+04  &5$|$5  \\
	300&  0.02& 163.94 $|$106.42 &  4.05e-05$|$2.05e-04  & 3.35e+03 $|$ 7.93e+04  &5$|$5  \\
	     &0.04&  164.67 $|$ 108.17 &  4.52e-05$|$2.20e-04  & 4.03e+03 $|$ 8.52e+04  &5$|$5  \\
         &0.06& 165.34 $|$ 110.52 &  4.62e-05$|$2.09e-04  & 3.86e+03 $|$ 8.41e+04  &5$|$5  \\
	     &0.08& 165.40 $|$110.46 &  4.32e-05$|$2.13e-04  & 5.01e+03 $|$ 9.03e+04  &5$|$5  \\
	      &0.1& 166.29 $|$112.09 &  4.59e-05$|$2.54e-04  & 5.60e+03 $|$ 7.97e+04  &5$|$20  \\
		\hline 	
	\end{tabular}
}
\end{table}

To observe the convergence behavior of each algorithm numerically, we report the final objective values (obj), the number of identified targets (n\_{tar}), the computing time required (Time), and the final KKT residuals ($\eta_{\max}$) in the case of BER=$\{0,0.02,0.04,0.06,0.08,0.1\}$. In this test, we only record these values of `n\_star', `Time', `$\eta_{\max}$', and `obj' when the step (step) reaches $100$, $200$, and $300$. These results are given in Table \ref{tab:maxIter123} where `sGSA' and `dA' are short for `sGS-ADMM' and `ADMM', respectively. From the last column, we see that both algorithms could successfully detect the number of targets, and from the last 2nd to 3rd  columns, we see that sGS-ADMM outperforms ADMM in the sense that it derives higher quality solutions with lower final objective function values and smaller KKT residuals. But, we also see that sGS-ADMM always requires more computing time at each test case. I think that this phenomenon is not surprising because sGS-ADMM computes some variables twice per-iteration.  
To visibly see the performance of each algorithm, we plot the KKT residuals' behavior versus the iteration  in Figure \ref{fig:eta}.
We see from this figure that, the KKT residuals generated by sGS-ADMM are always at the bottom, which indicates that sGS-ADMM decreases faster than ADMM. 
In conclusion, this test shows that both the objective-function values and the KKT residuals  derived by sGS-ADMM decrease faster than that obtained by ADMM as iteration goes on.
\begin{figure}[h]
	\centering
	\includegraphics[width=3in]{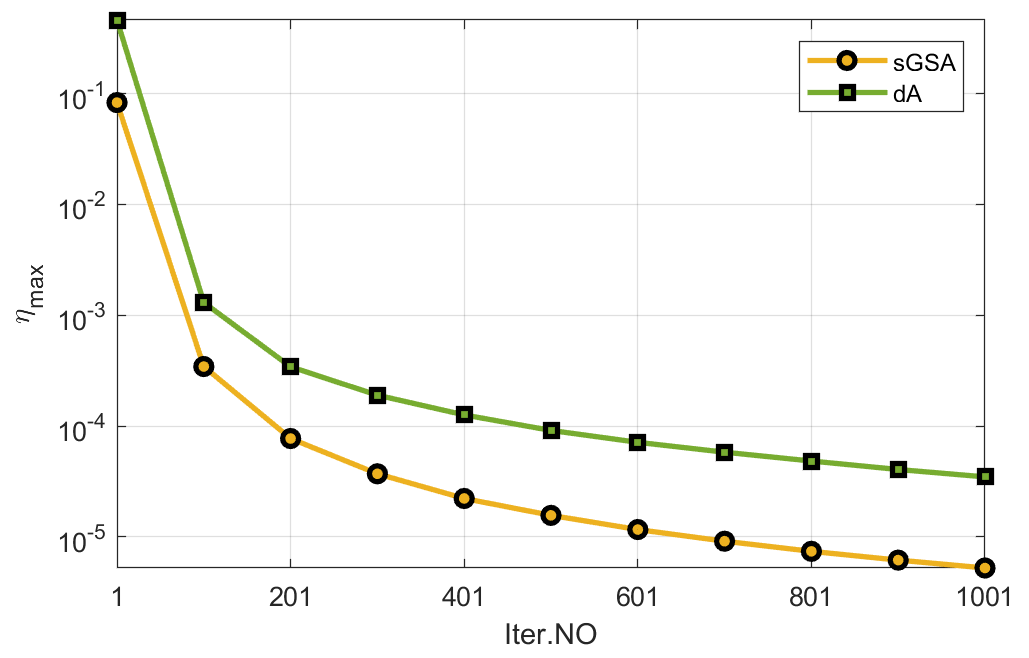}
	\caption{The KKT residual with iteration number increasing}
	\label{fig:eta}
\end{figure}

To compare the numerical performance in a relatively fair way, we run both algorithms from the same point and terminate when their KKT residuals are smaller than a given tolerance, i.e.,
$$
\eta_{\max}\leq \text{Tol}.
$$
In this test, we choose the tolerance as Tol=$1e-3$, $1e-4$, and $1e-5$, respectively.
Besides, we also force the iterative process to terminate when the number of iterations exceeds  $1000$ without achieving convergence. In this case, we say that the corresponding algorithm is failed. 
The detailed results derived by sGS-ADMM and ADMM with different BER values are reported in Table \ref{tab:kkt}, in which the symbol `---'  represents that the  algorithm fails to solve the corresponding problems.%最多500步
\begin{table}
	\centering
	\caption{Comparison results by sGS-ADMM and ADMM }
	\label{tab:kkt}
	{\small
		\setlength{\tabcolsep}{3pt}
		\begin{tabular}{ c| c c c c c c}
			\hline
			\multirow{2}{*}{BER} &	tol & Time(s) &$\eta_{\max}$  &obj &Iter.NO &$n_{tar}$
			\\ 
			&(sGSA$|$dA)	&  (sGSA$|$dA) &  (sGSA$|$dA) & (sGSA$|$dA)& (sGSA$|$dA)& (sGSA$|$dA)\\
			\toprule[1pt]
			& 1.0e-03 & 15.44$|$40.61   &9.88e-04$|$9.67e-04  & 9.16e+03$|$7.23e+04  & 29$|$114& 5$|$5   \\
			0    & 1.0e-04 & 95.32$|$163.41  &9.95e-05$|$9.97e-05  & 3.65e+03$|$7.30e+04  & 173$|$467 & 5$|$5 \\
			& 1.0e-05 & 361.44$|$342.50 &9.99e-06$|$3.64e-05  & 2.15e+03$|$7.32e+04  & 663$|$--- & 5$|$5 \\
			\hline 
			& 1.0e-03 & 14.93$|$44.19   &9.78e-04$|$9.85e-04  & 9.46e+04$|$7.80e+04  & 28$|$126 & 5$|$5  \\
			0.02  & 1.0e-04 & 95.46$|$173.06  &9.94e-05$|$9.98e-05  & 4.01e+03$|$7.92e+04  & 178$|$483 & 5$|$5  \\
			& 1.0e-05 & 368.34$|$349.95 &9.98e-06$|$3.50e-05  & 2.53e+03$|$7.95e+04  & 673$|$--- & 5$|$5  \\
			\hline
			& 1.0e-03 & 14.89$|$46.54   &9.84e-04$|$9.81e-04  & 9.90e+03$|$8.23e+04  & 27$|$130 & 5$|$5 \\
			0.04  & 1.0e-04 & 99.69$|$175.70  &9.99e-05$|$9.98e-05  & 4.41e+03$|$8.41e+04  & 185$|$500 & 5$|$5\\
			& 1.0e-05 & 379.58$|$351.64 &9.98e-06$|$3.65e-05  & 3.01e+03$|$8.45e+04  & 694$|$--- & 5$|$5 \\
			\hline
			& 1.0e-03 &  11.23$|$48.49  &9.81e-04$|$9.84e-04  & 1.10e+04$|$8.40e+04  & 21$|$32 & 2$|$5   \\
			0.06  & 1.0e-04 & 106.24$|$174.42 &9.93e-05$|$9.99e-05  & 4.90e+03$|$8.52e+04  & 191$|$479 & 5$|$5 \\
			& 1.0e-05 & 397.01$|$362.69 &9.99e-06$|$3.77e-05  & 3.45e+03$|$8.60e+04  & 719$|$---  & 5$|$5 \\
			\hline
			& 1.0e-03 & 11.28$|$49.32   &9.75e-04$|$9.80e-04  & 1.12e+04$|$8.73e+04  & 21$|$135 & 1$|$5   \\
			0.08  & 1.0e-04 & 109.69$|$193.14 &9.95e-05$|$9.97e-05  & 5.37e+03$|$8.59e+04  & 201$|$525 & 5$|$6 \\
			& 1.0e-05 & 432.70$|$368.50 &9.97e-06$|$3.81e-05  & 3.69e+03$|$8.52e+04  & 798$|$--- & 5$|$5 \\
			\hline
			& 1.0e-03 & 9.04$|$51.60    &9.71e-04$|$9.87e-04  & 1.22e+04$|$9.06e+04  & 17$|$141 & 1$|$15   \\
			0.1   & 1.0e-04 & 107.59$|$222.75 &9.93e-05$|$9.98e-05  & 6.04e+03$|$8.51e+04  & 196$|$602 & 5$|$22 \\
			& 1.0e-05 & 422.72$|$368.05 &9.98e-06$|$5.52e-05  & 3.98e+03$|$8.28e+04  & 776$|$--- & 5$|$30\\
			\toprule[1pt]
		\end{tabular}
	}
\end{table}

We see from  Table \ref{tab:kkt} that,  sGS-ADMM is capable of solving all the problems in each tested case.
However, ADMM still fails in the Tol=$1e-5$ case although it is successful in the low- and mid-accuracy cases. 
In addition, we also notice that in the case that both algorithms succeeded, sGS-ADMM is still two times faster. 
The better performance of sGS-ADMM  is consistent with the theoretical analysis aforementioned that the sGS has the ability to ensure   convergence. While turning our attention to the final objective  function values and the final KKT residuals, we find that these values derived by sGS-ADMM are always smaller than that by ADMM, which once again indicates that sGS-ADMM is a winner. 
At last, we see from the last column of this table that the number of targets identified by ADMM is not correct, especially in the case of BER$=0.1$,  because the true number of moving targets is only $5$.  
All in all, this test demonstrates that sGS-ADMM outputs a higher recognition rate and  a lower false alarm rate. It should be noted that, since the targets are estimated by means of grid search but  do not  fall onto the discrete grids,  the error related to the spacing of grids may exist invariably.

\section{Conclusions and remarks}

In recent literature, it was known that estimating the
joint delay-Doppler of moving targets in OFDM passive radar can be described as an atomic-norm regularized convex minimization problem, or a SDP minimization problem.
To solve this SDP problem, the directly-extended ADMM  is not necessarily convergent although it performed well experimentally. 
To address this issue, in this paper, we skillfully partitioned the variables $({\bm{e}},\bm{g},{\bm{z}}, \epsilon,\bm U,\bm{\Theta})$ into two groups, i.e., $(\bm e, \bm g)$ and $(\bm z,   \epsilon,\bm U,\bm\Theta)$, and then used sGS at each group.
One may think that this iteration form may lead to more computational costs thus leading to the algorithm being much more inefficient. 
But fortunately, the convergence result listed in Theorem \ref{the31} claimed that this algorithm should be more robust and efficient because the convergence property indeed reduced the number of iterations greatly.
This claim was verified experimentally, that is, the sGS-ADMM worked successfully in most tested cases, derived higher-quality solutions, and run at least two times faster than ADMM.

There are some interesting topics that deserve further investigating. Firstly, the superior performance of sGS-ADMM is confirmed using some simulated data. But, its practical behaviors on real data sets need more testing. 
Secondly, we see that all the subproblems in sGS-ADMM are solved exactly. Hence, some strategies to solve these subproblems inexactly but has the ability to ensure convergence are also essential.
Thirdly, we note that our approach is not able to locate the interesting targets accurately according to the velocity and range. Hence, the task of joint estimation of the velocity, range, and direction needs further investigation.
At last, it is also an interesting topic to develop other efficient algorithms, such as algorithms for dual problem, to improve the performance of sGS-ADMM.

\section*{Acknowledges}%\label{sec8}
The work of Y. Xiao is supported by the National Natural Science Foundation of China (Grants No. 11971149 and 12271217).

\section*{Conflict of interest}%\label{sec8}
No potential conflict of interest was reported by the authors.

\bibliographystyle{elsarticle-harv}
%elsarticle-num
%elsarticle-harv
%elsarticle-num-names
\bibliography{ref}

\begin{thebibliography}{36}
\expandafter\ifx\csname natexlab\endcsname\relax\def\natexlab#1{#1}\fi
\providecommand{\url}[1]{\texttt{#1}}
\providecommand{\href}[2]{#2}
\providecommand{\path}[1]{#1}
\providecommand{\DOIprefix}{doi:}
\providecommand{\ArXivprefix}{arXiv:}
\providecommand{\URLprefix}{URL: }
\providecommand{\Pubmedprefix}{pmid:}
\providecommand{\doi}[1]{\href{http://dx.doi.org/#1}{\path{#1}}}
\providecommand{\Pubmed}[1]{\href{pmid:#1}{\path{#1}}}
\providecommand{\bibinfo}[2]{#2}
\ifx\xfnm\relax \def\xfnm[#1]{\unskip,\space#1}\fi
%Type = Article
\bibitem[{Baczyk and Malanowski(2011)}]{baczyk2011reconstruction}
\bibinfo{author}{Baczyk, M.}, \bibinfo{author}{Malanowski, M.},
  \bibinfo{year}{2011}.
\newblock \bibinfo{title}{Reconstruction of the reference signal in dvb-t-based
  passive radar}.
\newblock \bibinfo{journal}{International Journal of Electronics and
  Telecommunications} .
%Type = Article
\bibitem[{Bajwa et~al.(2010)Bajwa, Haupt, Sayeed and
  Nowak}]{bajwa2010compressed}
\bibinfo{author}{Bajwa, W.U.}, \bibinfo{author}{Haupt, J.},
  \bibinfo{author}{Sayeed, A.M.}, \bibinfo{author}{Nowak, R.},
  \bibinfo{year}{2010}.
\newblock \bibinfo{title}{Compressed channel sensing: A new approach to
  estimating sparse multipath channels}.
\newblock \bibinfo{journal}{Proceedings of the IEEE} \bibinfo{volume}{98},
  \bibinfo{pages}{1058--1076}.
%Type = Article
\bibitem[{Berger et~al.(2010)Berger, Demissie, Heckenbach, Willett and
  Zhou}]{berger2010signal}
\bibinfo{author}{Berger, C.R.}, \bibinfo{author}{Demissie, B.},
  \bibinfo{author}{Heckenbach, J.}, \bibinfo{author}{Willett, P.},
  \bibinfo{author}{Zhou, S.}, \bibinfo{year}{2010}.
\newblock \bibinfo{title}{Signal processing for passive radar using ofdm
  waveforms}.
\newblock \bibinfo{journal}{IEEE Journal of Selected Topics in Signal
  Processing} \bibinfo{volume}{4}, \bibinfo{pages}{226--238}.
%Type = Article
\bibitem[{Bhaskar et~al.(2013)Bhaskar, Tang and Recht}]{bhaskar2013atomic}
\bibinfo{author}{Bhaskar, B.N.}, \bibinfo{author}{Tang, G.},
  \bibinfo{author}{Recht, B.}, \bibinfo{year}{2013}.
\newblock \bibinfo{title}{Atomic norm denoising with applications to line
  spectral estimation}.
\newblock \bibinfo{journal}{IEEE Transactions on Signal Processing}
  \bibinfo{volume}{61}, \bibinfo{pages}{5987--5999}.
%Type = Article
\bibitem[{Cand{\`e}s and Fernandez-Granda(2013)}]{candes2013super}
\bibinfo{author}{Cand{\`e}s, E.}, \bibinfo{author}{Fernandez-Granda, C.},
  \bibinfo{year}{2013}.
\newblock \bibinfo{title}{Super-resolution from noisy data}.
\newblock \bibinfo{journal}{Journal of Fourier Analysis and Applications}
  \bibinfo{volume}{19}, \bibinfo{pages}{1229--1254}.
\newblock \DOIprefix\doi{10.1007/s00041-013-9292-3}. \bibinfo{note}{copyright:
  Copyright 2013 Elsevier B.V., All rights reserved.}
%Type = Article
\bibitem[{Cand{\`e}s and Fernandez-Granda(2014)}]{candes2014towards}
\bibinfo{author}{Cand{\`e}s, E.J.}, \bibinfo{author}{Fernandez-Granda, C.},
  \bibinfo{year}{2014}.
\newblock \bibinfo{title}{Towards a mathematical theory of super-resolution}.
\newblock \bibinfo{journal}{Communications on pure and applied Mathematics}
  \bibinfo{volume}{67}, \bibinfo{pages}{906--956}.
%Type = Inproceedings
\bibitem[{Cardinali et~al.(2007)Cardinali, Colone, Ferretti and
  Lombardo}]{cardinali2007comparison}
\bibinfo{author}{Cardinali, R.}, \bibinfo{author}{Colone, F.},
  \bibinfo{author}{Ferretti, C.}, \bibinfo{author}{Lombardo, P.},
  \bibinfo{year}{2007}.
\newblock \bibinfo{title}{Comparison of clutter and multipath cancellation
  techniques for passive radar}, in: \bibinfo{booktitle}{2007 IEEE Radar
  Conference}, \bibinfo{organization}{IEEE}. pp. \bibinfo{pages}{469--474}.
%Type = Article
\bibitem[{Chandrasekaran et~al.(2012)Chandrasekaran, Recht, Parrilo and
  Willsky}]{chandrasekaran2012convex}
\bibinfo{author}{Chandrasekaran, V.}, \bibinfo{author}{Recht, B.},
  \bibinfo{author}{Parrilo, P.A.}, \bibinfo{author}{Willsky, A.S.},
  \bibinfo{year}{2012}.
\newblock \bibinfo{title}{The convex geometry of linear inverse problems}.
\newblock \bibinfo{journal}{Foundations of Computational mathematics}
  \bibinfo{volume}{12}, \bibinfo{pages}{805--849}.
%Type = Article
\bibitem[{Chen et~al.(2016)Chen, He, Ye and Yuan}]{chen2016direct}
\bibinfo{author}{Chen, C.}, \bibinfo{author}{He, B.}, \bibinfo{author}{Ye, Y.},
  \bibinfo{author}{Yuan, X.}, \bibinfo{year}{2016}.
\newblock \bibinfo{title}{The direct extension of admm for multi-block convex
  minimization problems is not necessarily convergent}.
\newblock \bibinfo{journal}{Mathematical Programming} \bibinfo{volume}{155},
  \bibinfo{pages}{57--79}.
%Type = Article
\bibitem[{Chen et~al.(2017)Chen, Sun and Toh}]{chen2017efficient}
\bibinfo{author}{Chen, L.}, \bibinfo{author}{Sun, D.}, \bibinfo{author}{Toh,
  K.C.}, \bibinfo{year}{2017}.
\newblock \bibinfo{title}{An efficient inexact symmetric gauss--seidel based
  majorized admm for high-dimensional convex composite conic programming}.
\newblock \bibinfo{journal}{Mathematical Programming} \bibinfo{volume}{161},
  \bibinfo{pages}{237--270}.
%Type = Inproceedings
\bibitem[{Colone et~al.(2006)Colone, Cardinali and
  Lombardo}]{colone2006cancellation}
\bibinfo{author}{Colone, F.}, \bibinfo{author}{Cardinali, R.},
  \bibinfo{author}{Lombardo, P.}, \bibinfo{year}{2006}.
\newblock \bibinfo{title}{Cancellation of clutter and multipath in passive
  radar using a sequential approach}, in: \bibinfo{booktitle}{2006 IEEE
  Conference on Radar}, \bibinfo{organization}{IEEE}. pp.
  \bibinfo{pages}{1--7}.
%Type = Article
\bibitem[{Colone et~al.(2009)Colone, O'hagan, Lombardo and
  Baker}]{colone2009multistage}
\bibinfo{author}{Colone, F.}, \bibinfo{author}{O'hagan, D.},
  \bibinfo{author}{Lombardo, P.}, \bibinfo{author}{Baker, C.},
  \bibinfo{year}{2009}.
\newblock \bibinfo{title}{A multistage processing algorithm for disturbance
  removal and target detection in passive bistatic radar}.
\newblock \bibinfo{journal}{IEEE Transactions on aerospace and electronic
  systems} \bibinfo{volume}{45}, \bibinfo{pages}{698--722}.
%Type = Article
\bibitem[{Ding and Xiao(2018)}]{ding2018symmetric}
\bibinfo{author}{Ding, Y.}, \bibinfo{author}{Xiao, Y.}, \bibinfo{year}{2018}.
\newblock \bibinfo{title}{Symmetric gauss--seidel technique-based alternating
  direction methods of multipliers for transform invariant low-rank textures
  problem}.
\newblock \bibinfo{journal}{Journal of Mathematical Imaging and Vision}
  \bibinfo{volume}{60}, \bibinfo{pages}{1220--1230}.
%Type = Inproceedings
\bibitem[{Falcone et~al.(2010)Falcone, Colone, Bongioanni and
  Lombardo}]{falcone2010experimental}
\bibinfo{author}{Falcone, P.}, \bibinfo{author}{Colone, F.},
  \bibinfo{author}{Bongioanni, C.}, \bibinfo{author}{Lombardo, P.},
  \bibinfo{year}{2010}.
\newblock \bibinfo{title}{Experimental results for ofdm wifi-based passive
  bistatic radar}, in: \bibinfo{booktitle}{2010 ieee radar conference},
  \bibinfo{organization}{IEEE}. pp. \bibinfo{pages}{516--521}.
%Type = Article
\bibitem[{Fazel et~al.(2013)Fazel, Pong, Sun and Tseng}]{fazel2013hankel}
\bibinfo{author}{Fazel, M.}, \bibinfo{author}{Pong, T.K.},
  \bibinfo{author}{Sun, D.}, \bibinfo{author}{Tseng, P.}, \bibinfo{year}{2013}.
\newblock \bibinfo{title}{Hankel matrix rank minimization with applications to
  system identification and realization}.
\newblock \bibinfo{journal}{SIAM Journal on Matrix Analysis and Applications}
  \bibinfo{volume}{34}, \bibinfo{pages}{946--977}.
%Type = Article
\bibitem[{Herman and Strohmer(2009)}]{herman2009high}
\bibinfo{author}{Herman, M.A.}, \bibinfo{author}{Strohmer, T.},
  \bibinfo{year}{2009}.
\newblock \bibinfo{title}{High-resolution radar via compressed sensing}.
\newblock \bibinfo{journal}{IEEE transactions on signal processing}
  \bibinfo{volume}{57}, \bibinfo{pages}{2275--2284}.
%Type = Article
\bibitem[{Howland et~al.(2005)Howland, Maksimiuk and Reitsma}]{howland2005fm}
\bibinfo{author}{Howland, P.E.}, \bibinfo{author}{Maksimiuk, D.},
  \bibinfo{author}{Reitsma, G.}, \bibinfo{year}{2005}.
\newblock \bibinfo{title}{Fm radio based bistatic radar}.
\newblock \bibinfo{journal}{IEE proceedings-radar, sonar and navigation}
  \bibinfo{volume}{152}, \bibinfo{pages}{107--115}.
%Type = Article
\bibitem[{Hu et~al.(2012)Hu, Shi, Zhou and Fu}]{hu2012compressed}
\bibinfo{author}{Hu, L.}, \bibinfo{author}{Shi, Z.}, \bibinfo{author}{Zhou,
  J.}, \bibinfo{author}{Fu, Q.}, \bibinfo{year}{2012}.
\newblock \bibinfo{title}{Compressed sensing of complex sinusoids: An approach
  based on dictionary refinement}.
\newblock \bibinfo{journal}{IEEE Transactions on Signal Processing}
  \bibinfo{volume}{60}, \bibinfo{pages}{3809--3822}.
%Type = Inproceedings
\bibitem[{Ketpan et~al.(2015)Ketpan, Phonsri, Qian and
  Sellathurai}]{ketpan2015target}
\bibinfo{author}{Ketpan, W.}, \bibinfo{author}{Phonsri, S.},
  \bibinfo{author}{Qian, R.}, \bibinfo{author}{Sellathurai, M.},
  \bibinfo{year}{2015}.
\newblock \bibinfo{title}{On the target detection in ofdm passive radar using
  music and compressive sensing}, in: \bibinfo{booktitle}{2015 Sensor Signal
  Processing for Defence (SSPD)}, \bibinfo{organization}{IEEE}. pp.
  \bibinfo{pages}{1--5}.
%Type = Article
\bibitem[{Krim and Viberg(1996)}]{krim1996two}
\bibinfo{author}{Krim, H.}, \bibinfo{author}{Viberg, M.}, \bibinfo{year}{1996}.
\newblock \bibinfo{title}{Two decades of array signal processing research: the
  parametric approach}.
\newblock \bibinfo{journal}{IEEE signal processing magazine}
  \bibinfo{volume}{13}, \bibinfo{pages}{67--94}.
%Type = Article
\bibitem[{Li and Xiao(2018)}]{li2018efficient}
\bibinfo{author}{Li, P.}, \bibinfo{author}{Xiao, Y.}, \bibinfo{year}{2018}.
\newblock \bibinfo{title}{An efficient algorithm for sparse inverse covariance
  matrix estimation based on dual formulation}.
\newblock \bibinfo{journal}{Computational Statistics \& Data Analysis}
  \bibinfo{volume}{128}, \bibinfo{pages}{292--307}.
%Type = Article
\bibitem[{Li et~al.(2016)Li, Sun and Toh}]{li2016schur}
\bibinfo{author}{Li, X.}, \bibinfo{author}{Sun, D.}, \bibinfo{author}{Toh,
  K.C.}, \bibinfo{year}{2016}.
\newblock \bibinfo{title}{A schur complement based semi-proximal admm for
  convex quadratic conic programming and extensions}.
\newblock \bibinfo{journal}{Mathematical Programming} \bibinfo{volume}{155},
  \bibinfo{pages}{333--373}.
%Type = Article
\bibitem[{Li et~al.(2019)Li, Sun and Toh}]{li2019block}
\bibinfo{author}{Li, X.}, \bibinfo{author}{Sun, D.}, \bibinfo{author}{Toh,
  K.C.}, \bibinfo{year}{2019}.
\newblock \bibinfo{title}{A block symmetric gauss--seidel decomposition theorem
  for convex composite quadratic programming and its applications}.
\newblock \bibinfo{journal}{Mathematical Programming} \bibinfo{volume}{175},
  \bibinfo{pages}{395--418}.
%Type = Article
\bibitem[{Moreau(1962)}]{moreau1962fonctions}
\bibinfo{author}{Moreau, J.J.}, \bibinfo{year}{1962}.
\newblock \bibinfo{title}{Fonctions convexes duales et points proximaux dans un
  espace hilbertien}.
\newblock \bibinfo{journal}{Comptes rendus hebdomadaires des s{\'e}ances de
  l'Acad{\'e}mie des sciences} \bibinfo{volume}{255},
  \bibinfo{pages}{2897--2899}.
%Type = Article
\bibitem[{Naha et~al.(2014)Naha, Samanta, Routray and
  Deb}]{naha2014determining}
\bibinfo{author}{Naha, A.}, \bibinfo{author}{Samanta, A.K.},
  \bibinfo{author}{Routray, A.}, \bibinfo{author}{Deb, A.K.},
  \bibinfo{year}{2014}.
\newblock \bibinfo{title}{Determining autocorrelation matrix size and sampling
  frequency for music algorithm}.
\newblock \bibinfo{journal}{IEEE Signal Processing Letters}
  \bibinfo{volume}{22}, \bibinfo{pages}{1016--1020}.
%Type = Article
\bibitem[{Palmer et~al.(2012)Palmer, Harms, Searle and Davis}]{palmer2012dvb}
\bibinfo{author}{Palmer, J.E.}, \bibinfo{author}{Harms, H.A.},
  \bibinfo{author}{Searle, S.J.}, \bibinfo{author}{Davis, L.},
  \bibinfo{year}{2012}.
\newblock \bibinfo{title}{Dvb-t passive radar signal processing}.
\newblock \bibinfo{journal}{IEEE transactions on Signal Processing}
  \bibinfo{volume}{61}, \bibinfo{pages}{2116--2126}.
%Type = Article
\bibitem[{Sen and Nehorai(2010)}]{sen2010adaptive}
\bibinfo{author}{Sen, S.}, \bibinfo{author}{Nehorai, A.}, \bibinfo{year}{2010}.
\newblock \bibinfo{title}{Adaptive ofdm radar for target detection in multipath
  scenarios}.
\newblock \bibinfo{journal}{IEEE Transactions on Signal Processing}
  \bibinfo{volume}{59}, \bibinfo{pages}{78--90}.
%Type = Article
\bibitem[{Sun et~al.(2015)Sun, Toh and Yang}]{sun2015convergent}
\bibinfo{author}{Sun, D.}, \bibinfo{author}{Toh, K.C.}, \bibinfo{author}{Yang,
  L.}, \bibinfo{year}{2015}.
\newblock \bibinfo{title}{A convergent 3-block semiproximal alternating
  direction method of multipliers for conic programming with 4-type
  constraints}.
\newblock \bibinfo{journal}{SIAM journal on Optimization} \bibinfo{volume}{25},
  \bibinfo{pages}{882--915}.
%Type = Article
\bibitem[{Tang et~al.(2013)Tang, Bhaskar, Shah and Recht}]{tang2013compressed}
\bibinfo{author}{Tang, G.}, \bibinfo{author}{Bhaskar, B.N.},
  \bibinfo{author}{Shah, P.}, \bibinfo{author}{Recht, B.},
  \bibinfo{year}{2013}.
\newblock \bibinfo{title}{Compressed sensing off the grid}.
\newblock \bibinfo{journal}{IEEE transactions on information theory}
  \bibinfo{volume}{59}, \bibinfo{pages}{7465--7490}.
%Type = Article
\bibitem[{Tao et~al.(2010)Tao, Wu and Shan}]{tao2010direct}
\bibinfo{author}{Tao, R.}, \bibinfo{author}{Wu, H.}, \bibinfo{author}{Shan,
  T.}, \bibinfo{year}{2010}.
\newblock \bibinfo{title}{Direct-path suppression by spatial filtering in
  digital television terrestrial broadcasting-based passive radar}.
\newblock \bibinfo{journal}{IET radar, sonar \& navigation}
  \bibinfo{volume}{4}, \bibinfo{pages}{791--805}.
%Type = Article
\bibitem[{Yang and Xie(2016)}]{yang2016exact}
\bibinfo{author}{Yang, Z.}, \bibinfo{author}{Xie, L.}, \bibinfo{year}{2016}.
\newblock \bibinfo{title}{Exact joint sparse frequency recovery via
  optimization methods}.
\newblock \bibinfo{journal}{IEEE Transactions on Signal Processing}
  \bibinfo{volume}{64}, \bibinfo{pages}{5145--5157}.
%Type = Article
\bibitem[{Yang et~al.(2016)Yang, Xie and Stoica}]{yang2016vandermonde}
\bibinfo{author}{Yang, Z.}, \bibinfo{author}{Xie, L.}, \bibinfo{author}{Stoica,
  P.}, \bibinfo{year}{2016}.
\newblock \bibinfo{title}{Vandermonde decomposition of multilevel toeplitz
  matrices with application to multidimensional super-resolution}.
\newblock \bibinfo{journal}{IEEE Transactions on Information Theory}
  \bibinfo{volume}{62}, \bibinfo{pages}{3685--3701}.
%Type = Article
\bibitem[{Yang et~al.(2012a)Yang, Xie and Zhang}]{yang2012off}
\bibinfo{author}{Yang, Z.}, \bibinfo{author}{Xie, L.}, \bibinfo{author}{Zhang,
  C.}, \bibinfo{year}{2012}a.
\newblock \bibinfo{title}{Off-grid direction of arrival estimation using sparse
  bayesian inference}.
\newblock \bibinfo{journal}{IEEE transactions on signal processing}
  \bibinfo{volume}{61}, \bibinfo{pages}{38--43}.
%Type = Article
\bibitem[{Yang et~al.(2012b)Yang, Zhang and Xie}]{yang2012robustly}
\bibinfo{author}{Yang, Z.}, \bibinfo{author}{Zhang, C.}, \bibinfo{author}{Xie,
  L.}, \bibinfo{year}{2012}b.
\newblock \bibinfo{title}{Robustly stable signal recovery in compressed sensing
  with structured matrix perturbation}.
\newblock \bibinfo{journal}{IEEE Transactions on Signal Processing}
  \bibinfo{volume}{60}, \bibinfo{pages}{4658--4671}.
%Type = Article
\bibitem[{Zhang et~al.(2015)Zhang, Li, Liu and Himed}]{zhang2015joint}
\bibinfo{author}{Zhang, X.}, \bibinfo{author}{Li, H.}, \bibinfo{author}{Liu,
  J.}, \bibinfo{author}{Himed, B.}, \bibinfo{year}{2015}.
\newblock \bibinfo{title}{Joint delay and doppler estimation for passive
  sensing with direct-path interference}.
\newblock \bibinfo{journal}{IEEE Transactions on Signal Processing}
  \bibinfo{volume}{64}, \bibinfo{pages}{630--640}.
%Type = Article
\bibitem[{Zheng and Wang(2017)}]{zheng2017super}
\bibinfo{author}{Zheng, L.}, \bibinfo{author}{Wang, X.}, \bibinfo{year}{2017}.
\newblock \bibinfo{title}{Super-resolution delay-doppler estimation for ofdm
  passive radar}.
\newblock \bibinfo{journal}{IEEE Transactions on Signal Processing}
  \bibinfo{volume}{65}, \bibinfo{pages}{2197--2210}.

\end{thebibliography}

\end{document}